\documentclass[12pt]{article}
\usepackage{amssymb,amsmath,amsthm,amsfonts,verbatim}
\usepackage{youngtab}
\usepackage{graphicx}
\newtheorem{theorem}{Theorem}
\newtheorem*{remark}{Remark}
\newtheorem*{fact}{Fact}
\newtheorem*{definition}{Definition}
\newtheorem{proposition}[theorem]{Proposition}
\newtheorem{conjecture}{Conjecture}
\newtheorem*{question}{Question}
\newtheorem{lemma}[theorem]{Lemma}
\newtheorem{claim}{Claim}
\newtheorem{corollary}[theorem]{Corollary}

\DeclareMathOperator{\constant}{const}

\newcommand{\Trace}{\textrm{Tr}}
\newcommand{\const}{\textrm{const}}

\newcommand{\sgn}{\textrm{sgn}}
\newcommand{\Span}{\textrm{Span}}

\newcommand{\Ker}{\textrm{Ker}}

\newcommand{\domleq}{\unlhd}
\newcommand{\domgeq}{\unrhd}

\newcommand{\Cay}{\textrm{Cay}}

\newcommand{\round}[1]{\operatorname{round}(#1)}

\newcommand{\A}[0]{{\cal A}}

\begin{document}
\title{Setwise intersecting families of permutations}
\author{David Ellis}
\date{June 2011}
\maketitle

\begin{abstract}
A family of permutations \(\mathcal{A} \subset S_{n}\) is said to be {\em \(t\)-set-intersecting} if for any two permutations \(\sigma,\pi \in \mathcal{A}\), there exists a \(t\)-set \(x\) whose image is the same under both permutations, i.e. \(\sigma(x)=\pi(x)\). We prove that if \(n\) is sufficiently large depending on \(t\), the largest \(t\)-set-intersecting families of permutations in \(S_n\) are cosets of stabilizers of \(t\)-sets. The \(t=2\) case of this was conjectured by J\'anos K\"orner. It can be seen as a variant of the Deza-Frankl conjecture, proved in \cite{jointpaper}. Our proof uses similar techniques to those of \cite{jointpaper}, namely, eigenvalue methods, together with the representation theory of the symmetric group, but the combinatorial part of the proof is harder.
\end{abstract}

\section{Introduction}
Intersection problems are of wide interest in combinatorics. For example, we say that a family of \(r\)-element subsets of \(\{1,2,\ldots,n\}\) is {\em intersecting} if any two of its sets have nonempty intersection. The classical Erd\H os-Ko-Rado theorem states that if $r < n/2$, an intersecting family of
 $r$-element subsets of $\{1,2,\ldots,n\}$ has size at most ${n-1 \choose r-1}$;
 if equality holds, the family must consist of all $r$-subsets containing a fixed element. We consider intersection problems for families of permutations.

Let \(S_n\) denote the symmetric group, the group of all permutations of \(\{1,2,\ldots,n\}\) under composition. A family of permutations \(\mathcal{A} \subset S_{n}\) is said to be \textit{intersecting} if any two permutations in \(\mathcal{A}\) agree at some point, i.e. for any \(\sigma, \pi \in \mathcal{A}\), there is some \(i \in [n]\) such that \(\sigma(i)=\pi(i)\). It is natural to ask, how large can an intersecting family be? An example of a large intersecting family of permutations is \(\{\sigma \in S_n:\ \sigma(1)=1\}\), the collection of all permutations fixing 1. In 1977, Deza and Frankl \cite{dezafrankl} showed that if \(\mathcal{A} \subset S_{n}\) is intersecting, then \(|\mathcal{A}| \leq (n-1)!\), using a simple partitioning argument. They conjectured that equality holds only if \(\mathcal{A}\) is a coset of the stabilizer of a point. This was proved independently by Cameron and Ku \cite{cameron} and Larose and Malvenuto \cite{larose}; Wang and Zhang \cite{wang} have recently given a shorter proof.

Deza and Frankl considered a natural generalization of this question. A family of permutations \(\mathcal{A} \subset S_{n}\) is said to be \(t\)-\textit{intersecting} if any two permutations in \(\mathcal{A}\) agree on at least \(t\) points, i.e., \(|\{i \in [n]: \sigma(i)=\pi(i)\}| \geq t\) for any \(\sigma, \pi \in \mathcal{A}\). An example of a large \(t\)-intersecting family is
\[\{\sigma \in S_n:\ \sigma(i)=i\ \forall i \in [t]\};\]
this has size \((n-t)!\). Deza and Frankl \cite{dezafrankl} conjectured that for \(n\) sufficiently large depending on \(t\), if \(\mathcal{A} \subset S_{n}\) is \(t\)-intersecting, then \(|\mathcal{A}| \leq (n-t)!\), with equality only if \(\mathcal{A}\) is a coset of the stabilizer of \(t\) points. This became known as the Deza-Frankl conjecture. It was proved in 2008 by the author and independently by Friedgut and Pilpel, using the same methods: eigenvalue techniques, together with the representation theory of the symmetric group. We have written a joint paper, \cite{jointpaper}.

In this paper, we consider another natural generalization of the original question. A family of permutations \(\mathcal{A} \subset S_{n}\) is said to be {\em\(t\)-set-intersecting} if for any two permutations \(\sigma, \pi \in \mathcal{A}\), there is some \(t\)-set \(x\) which has the same image-set under both permutations, i.e. \(\sigma(x)=\pi(x)\). In other words, if we consider the permutation action of \(\mathcal{A}\) on \([n]^{(t)}\), then any two permutations in \(\mathcal{A}\) agree on some \(t\)-set. We prove the following.

\begin{theorem}
 \label{thm:main}
If \(n\) is sufficiently large depending on \(t\), and \(\mathcal{A} \subset S_n\) is \(t\)-set-intersecting, then \(|\mathcal{A}| \leq t!(n-t)!\). Equality holds only if \(\mathcal{A}\) is a coset of the stabilizer of a \(t\)-set.
\end{theorem}

The \(t=2\) case of this was conjectured by K\"orner \cite{korner}. Our proof uses similar techniques to that of the Deza-Frankl conjecture in \cite{jointpaper} --- namely, eigenvalue techniques, together with the representation theory of the symmetric group --- but the arguments are slightly harder than in \cite{jointpaper}.

Of course, a 1-set-intersecting family is simply an intersecting family, so the \(t=1\) case of Theorem \ref{thm:main} follows from the theorems on intesecting families mentioned above. We briefly recall Deza and Frankl's `partitioning' proof that an intersecting family \(\mathcal{A} \subset S_n\) has size at most \((n-1)!\). Take any \(n\)-cycle \(\rho\), and let \(H\) be the cyclic group of order \(n\) generated by \(\rho\). For any left coset \(\sigma H\) of \(H\), any two distinct permutations in \(\sigma H\) disagree at every point, and therefore \(\sigma H\) contains at most one member of \(\mathcal{A}\). Since the left cosets of \(H\) partition \(S_{n}\), it follows that \(|\mathcal{A}| \leq (n-1)!\).

What happens if we try to generalize this argument to prove the upper bound in Theorem \ref{thm:main}? If \(T \subset S_n\), we say that \(T\) is {\em \(t\)-set-transitive} if for any \(x,y \in [n]^{(t)}\), there exists \(\sigma \in T\) such that \(\sigma(x)=y\). It is said to be {\em sharply \(t\)-set-transitive} if for any \(x,y \in [n]^{(t)}\), there exists a {\em unique} \(\sigma \in T\) such that \(\sigma(x)=y\). Note that a \(t\)-set-transitive subset \(T \subset S_n\) is sharply \(t\)-set-transitive if and only if \(|T|={n \choose t}\).

When there exists a sharply \(t\)-set-transitive {\em subgroup} of \(S_n\), Deza and Frankl's partitioning method yields the upper bound in Theorem \ref{thm:main}. When there exists a sharply \(t\)-set-transitive {\em subset} of \(S_n\), the  partitioning method can be replaced by a Katona-type averaging argument. Indeed, suppose that $S_{n}$ has a sharply $t$-set-transitive subset $T$. Then any left translate $\sigma T$ of $T$ is also sharply $t$-set-transitive, so for any two distinct permutations \(\pi,\tau \in \sigma T\), \(\pi(x) \neq \tau(x)\) for all \(x \in [n]^{(t)}\). Let \(\mathcal{A} \subset S_n\) be \(t\)-set-intersecting; then $|\mathcal{A} \cap (\sigma T)| \leq 1$ for each \(\sigma \in S_n\). Averaging over all \(\sigma \in S_n\) gives $|\mathcal{A}| \leq t!(n-t)!$.

However, sharply \(t\)-set-transitive subgroups of \(S_n\) only exist for a few special values of \(t\) and \(n\). When \(t=2\) and \({n \choose 2}\) is even, for example, there exists no sharply \(2\)-set-transitive subgroup of \(S_n\). Indeed, if \(H \leq S_n\) with \(|H| = {n \choose 2}\), then \(H\) contains a permutation \(\tau\) of order 2, which fixes some \(2\)-set \(x\), so \(\tau(x) = x = \textrm{id}(x)\), and \(H\) is not sharply \(2\)-transitive. Moreoever, when \(t=2\) and \(n=4\), it is easy to see that a \(2\)-set-intersecting family in \(S_4\) has size at most 4, but there is no sharply \(2\)-set-transitive subset of \(S_4\).

On the other hand, when \(t=2\) and \(n=5\), there {\em are} sharply \(2\)-set-transitive subsets of \(S_5\), for example, the following set of permutations, written in sequence notation:
\begin{align*}
1\ 2\ 3\ 4\ 5,\quad 2\ 3\ 4\ 5\ 1,\quad 3\ 4\ 5\ 1\ 2,\quad 4\ 5\ 1\ 2\ 3,\quad 5\ 1\ 2\ 3\ 4,\\
1\ 3\ 5\ 2\ 4,\quad 2\ 4\ 1\ 3\ 5,\quad 3\ 5\ 2\ 4\ 1,\quad 4\ 1\ 3\ 5\ 2,\quad 5\ 2\ 4\ 1\ 3.
\end{align*}
(Note that the 5 permutations on the first line are the cyclic shifts of \(1\ 2\ 3\ 4\ 5\), and the 5 on the second line are the cyclic shifts of \(1\ 3\ 5\ 2\ 4\).) It follows that a \(2\)-set-intersecting family in \(S_5\) has size at most \(12\).

While it may be that sharply \(t\)-set-transitive subsets of \(S_n\) exist for \(n\) sufficiently large depending upon \(t\), we can see no reason for this to be true. Hence, we will not explore this approach further.

\subsection*{Definitions and notation}
We pause to recall some standard definitions and notation. If \(n \in \mathbb{N}\), \([n]\) denotes the set \(\{1,2,\ldots,n\}\). If \(t \in \mathbb{N}\), \([n]^{(t)}\) denotes the set of all \(t\)-element subsets of \([n]\). If \(S\) is a set, the {\em characteristic function} \(1_S\) of \(S\) is defined by
\[1_{S}(i) = \left\{\begin{array}{ll} 1 & \textrm{if } i \in S,\\
                    0 & \textrm{if } i \notin S. \end{array}\right.\]
If \(H = (V,E)\) is a graph, an {\em independent set} in \(H\) is a set of vertices of \(H\) with no edges of \(H\) between them. The {\em independence number} of \(H\) is the maximum size of an independent set in \(H\).

Recall that if \(H = (V,E)\) is an \(n\)-vertex graph, the {\em adjacency matrix} \(A\) of \(H\) is defined by
\[A_{v,w} = \left\{\begin{array}{ll} 1 & \textrm{if } vw \in E(H),\\
                    0 & \textrm{if } vw \notin E(H).
                   \end{array}\right.\]
This is a real, symmetric, \(n \times n\) matrix, so there exists an orthonormal system of \(n\) eigenvectors of \(A\), which forms a basis for \(\mathbb{R}^V\). (Note that the eigenvalues of \(A\) are often referred to as the eigenvalues of \(H\).)

If \(G\) is a finite group, and \(S \subset G\) is inverse-closed (\(S^{-1}=S\)), the {\em Cayley graph on \(G\) generated by \(S\)} is the graph with vertex-set \(G\), where we join \(g\) to \(gs\) for each \(g \in G\) and each \(s \in S\); it is sometimes denoted \(\Cay(G,S)\). In other words,
\[V(\Cay(G,S)) = G,\quad E(\Cay(G,S)) = \{\{g,gs\}:\ g \in G,\ s \in S\}.\]
We say that \(\Cay(G,S)\) is a {\em normal} Cayley graph if \(S\) is conjugation-invariant, i.e. \(S\) is a union of conjugacy classes of \(G\).

As usual, we compose permutations from right to left, i.e. if \(\sigma,\pi \in S_n\) and \(i \in [n]\), \((\sigma \pi) (i) = \sigma(\pi(i))\).

For \(x,y \in [n]^{(t)}\), we write
\[T_{x \mapsto y} = \{\sigma \in S_n:\ \sigma(x)=y\}\]
for the set of all permutations mapping \(x\) to \(y\); the \(T_{x\mapsto y}\)'s are precisely the cosets of stablizers of \(t\)-sets. We will sometimes refer to them as the `\(t\)-SSCs', standing for \(t\)-set-stabilizer-cosets. Our aim is to show that these are the maximum-sized \(t\)-set-intersecting families in \(S_n\), if \(n\) is sufficiently large depending on \(t\).

As in \cite{jointpaper}, if \(a\) and \(b\) are \(t\)-tuples of distinct numbers between 1 and \(n\), we define
\[T_{a \mapsto b} = \{\sigma \in S_n:\ \sigma(a_1)=b_1,\ \sigma(a_2)=b_2,\ldots, \sigma(a_t)=b_t\};\]
the \(T_{a \mapsto b}\)'s are called the {\em \(t\)-cosets of \(S_n\)}. Note that a \(t\)-SSC is a disjoint union of \(t!\) \(t\)-cosets.

We say that a permutation \(\sigma \in S_n\) is a {\em derangement} if it has no fixed point. Similarly, we say that \(\sigma \in S_n\) \(t\)-{\em derangment} if it fixes no \(t\)-set, i.e. \(\sigma(x) \neq x\) for all \(x \in [n]^{(t)}\).

If \(\sigma\in S_n\), \(\sgn(\sigma)\) will denote the {\em sign} of \(\sigma\), which is \(1\) if \(\sigma\) is a product of an even number of transpositions, and \(-1\) otherwise.

Finally, if \(f(n,t)\) is a function of both \(n \in \mathbb{N}\) and \(t \in \mathbb{N}\), and \(g\) is a function of \(n \in \mathbb{N}\), we will write \(f(n,t) = O_t(g)\) to mean that for each fixed \(t\), \(f(n,t) = O(g)\) --- i.e., for each \(t\), there exists \(K_t\) (depending upon \(t\) alone) such that \(f(n,t) \leq K_t g(n)\) for all \(n \in \mathbb{N}\).

\subsection*{A sketch of the proof}
We now give a sketch of our proof. First, we rephrase the problem in terms of finding the maximum-sized independent sets in a certain graph.

Let \(\Gamma_{(t)}\) denote the graph with vertex-set \(S_n\), where two permutations are joined if no \(t\)-set has the same image under both permutations. In other words,
\[V(\Gamma_{(t)}) = S_n,\quad E(\Gamma_{(t)})=\{\{\sigma, \pi\}:\ \sigma(x) \neq \pi(x)\ \forall x \in [n]^{(t)}\}.\]
Let \(\mathcal{D}_{(t)}\) denote the set of all \(t\)-derangements in \(S_n\). Clearly, \(\Gamma_{(1)}\) is the Cayley graph on \(S_n\) generated by \(\mathcal{D}_{(1)}\), the set of all derangements in \(S_n\); it is often call the {\em derangement graph} on \(S_n\). Similarly, \(\Gamma_{(t)}\) is the Cayley graph on \(S_n\) generated by \(\mathcal{D}_{(t)}\), the set of all \(t\)-derangments in \(S_n\); we will call \(\Gamma_{(t)}\) the {\em \(t\)-derangment graph} on \(S_n\).

Observe that a \(t\)-set-intersecting family of permutations in \(S_n\) is precisely an independent set in \(\Gamma_{(t)}\). Hence, our task is to find the maximum-sized independent sets in the graph \(\Gamma_{(t)}\).

There are several well-known bounds on the independence number of a graph in terms of the eigenvalues of matrices related to the graph. The simplest is Hoffman's theorem \cite{hoffman}, which bounds the independence number of a \(d\)-regular graph in terms of the least eigenvalue of its adjacency matrix.

\begin{theorem}[Hoffman's theorem]
\label{thm:hoffman}
Let \(H = (V,E)\) be a \(d\)-regular, \(N\)-vertex graph. Let \(A\) be the adjacency matrix of \(H\). Let \(\lambda_{\min}\) denote the least eigenvalue of \(A\). If \(S \subset V\) is an independent set in \(H\), then
\[\frac{|S|}{N} \leq \frac{-\lambda_{\min}}{d - \lambda_{\min}}.\]
If equality holds, then the characteristic function \(1_{S}\) of \(S\) satisfies:
\[1_S - \tfrac{|S|}{N} \boldsymbol{1} \in \Ker(A - \lambda_{\min}I),\]
where \(\mathbf{1}\) denotes the all-\(1\)'s vector.
\end{theorem}

It is well-known that for any finite group \(G\), the vector space \(\mathbb{C}[G]\) of all complex-valued functions on \(G\) can be decomposed into a direct sum of subspaces which are eigenspaces of every normal Cayley graph on \(G\). (These subspaces are in 1-1 correspondence with the isomorphism classes of irreducible representations of \(G\).) Hence, the eigenvalues of normal Cayley graphs are relatively easy to handle.

It turns out that calculating the least eigenvalue of \(\Gamma_{(1)}\) (meaning the least eigenvalue of its adjacency matrix) and applying Hoffman's bound yields an alternative proof that an intersecting family in \(S_n\) has size at most \((n-1)!\). Calculating the least eigenvalue of \(\Gamma_{(1)}\) is non-trivial, requiring use of the representation theory of \(S_n\). It was first done by Renteln \cite{renteln}, using symmetric functions, and independently and slightly later by Friedgut and Pilpel \cite{friedguttalk}, and also by Godsil and Meagher \cite{meagher}. As observed in \cite{friedguttalk} and \cite{meagher}, this leads to an alternative proof that the maximum-sized intersecting families in \(S_n\) are cosets of stabilizers of points.

As in the case of the Deza-Frankl conjecture, the obvious generalization of this approach fails for \(t\)-set-intersecting families in \(S_n\): calculating the least eigenvalue of \(\Gamma_{(t)}\) and applying Hoffman's bound only gives an upper bound of \(\Theta((n-1)!)\) on the size of a \(t\)-set-intersecting family, for \(t\) fixed and \(n\) large.

Instead, we construct a `pseudoadjacency matrix' \(A\) for \(\Gamma_{(t)}\); this matrix \(A\) will be a suitable real linear combination of the adjacency matrices of certain subgraphs of \(\Gamma_{(t)}\). The eigenvalues of \(A\) will be such that when we apply a `weighted' version of Hoffman's theorem to the graph \(\Gamma_{(t)}\) (using the matrix \(A\) in place of the adjacency matrix), we obtain the desired upper bound of \(t!(n-t)!\), provided \(n\) is sufficiently large depending on \(t\). The subgraphs chosen will be normal Cayley graphs; this will enable us to analyse the eigenvalues of \(A\) using the representation theory of \(S_n\). Most of the work of the proof is in showing that a suitable linear combination exists; this turns out to be harder than in \cite{jointpaper}, where the same approach was used to prove the Deza-Frankl conjecture.

The eigenspaces of our pseudoadjacency matrix will be such as to imply that, for \(n\) sufficiently large depending on \(t\), if \(\mathcal{A} \subset S_n\) is a \(t\)-set-intersecting family of size \(t!(n-t)!\), then the characteristic function of \(\mathcal{A}\) is a linear combination of characteristic functions of \(t\)-SSC's (cosets of stabilizers of \(t\)-sets). By combining this fact with a result from \cite{jointpaper}, together with some additional combinatorial arguments, we will be able to show that \(\mathcal{A}\) must be a coset of the stabilizer of a \(t\)-set. This will complete the proof of Theorem \ref{thm:main}.

\section{Background}
\subsection*{The eigenvalue bound}
In this section, we recall the `weighted' version of Hoffman's bound we will need.
\begin{definition}
Let \(H = (V,E)\) be an \(N\)-vertex graph. A matrix \(A \in \mathbb{R}^{V \times V}\) is said to be a {\em pseudoadjacency matrix for \(H\)} if it is symmetric, has \(A_{x,y} = 0\) whenever \(xy \notin E(H)\), and has all its row and column sums equal.
\end{definition}

Note that the fact that all the row sums of \(A\) are equal implies that the constant functions are eigenvectors of \(A\). One has the following extension of Hoffman's theorem, due essentially to Delsarte:

\begin{theorem}
\label{thm:pseudo}
Let \(H = (V,E)\) be an \(N\)-vertex graph, and let \(A\) be a pseudoadjacency matrix for \(H\). Let \(\lambda_{\min}\) denote the least eigenvalue of \(A\), and let \(\lambda_{\constant}\) be the eigenvalue corresponding to the constant functions. If \(S \subset V\) is an independent set in \(H\), then
\[\frac{|S|}{N} \leq \frac{-\lambda_{\min}}{\lambda_{\constant} - \lambda_{\min}}.\]
If equality holds, then the characteristic function \(1_{S}\) of \(S\) satisfies:
\[1_S - \tfrac{|S|}{N} \boldsymbol{1} \in \Ker(A - \lambda_{\min}I).\]
\end{theorem}
The proof is a straightforward extension of that of Hoffman's theorem, and is given in \cite{jointpaper}.

\subsection*{Background on general representation theory}
In this section, we recall the basic notions and results we need from general representation theory. For more background, the reader may consult \cite{serre}.

Let \(G\) be a finite group. A {\em representation of \(G\) over \(\mathbb{C}\)} is a pair \((\rho,V)\), where \(V\) is a finite-dimensional vector space over \(\mathbb{C}\), and \(\rho:\ G \to GL(V)\) is a group homomorphism from \(G\) to the group of all invertible linear endomorphisms of \(V\). The vector space \(V\), together with the linear action of \(G\) defined by \(gv = \rho(g)(v)\), is sometimes called a \(\mathbb{C}G\)-{\em module}.  A {\em homomorphism} between two representations \((\rho,V)\) and \((\rho',V')\) is a linear map \(\phi:V \to V'\) such that such that \(\phi(\rho(g)(v)) = \rho'(g)(\phi(v))\) for all \(g \in G\) and \(v \in V\). If \(\phi\) is a linear isomorphism, the two representations are said to be {\em equivalent}, or {\em isomorphic}, and we write \((\rho,V) \cong (\rho',V')\). If \(\dim(V)=n\), we say that \(\rho\) {\em has dimension} \(n\). If \(V = \mathbb{C}^{n}\), then we call \(\rho\) a {\em matrix representation}; choosing a basis for a general \(V\), one sees that any representation is equivalent to some matrix representation.

The representation \((\rho,V)\) is said to be {\em irreducible} if it has no proper subrepresentation, i.e. there is no proper subspace of \(V\) which is \(\rho(g)\)-invariant for all \(g \in G\).

The {\em group algebra} \(\mathbb{C}G\) denotes the complex vector space with basis \(G\) and multiplication defined by extending the group multiplication linearly. In other words,
\[\mathbb{C}G = \left\{\sum_{g \in G}x_{g}g:\ x_{g} \in \mathbb{C}\ \forall g \in G\right\},\]
and
\[\left(\sum_{g \in G} x_{g}g\right)\left(\sum_{h\in G}y_{h}h\right) = \sum_{g,h \in G} x_{g}y_{h} (g h).\]
As a vector-space, \(\mathbb{C}G\) can be identified with \(\mathbb{C}[G]\), the vector space of all complex-valued functions on \(G\), by identifying \(\sum_{g \in G} x_g g\) with the function \(g \mapsto x_g\). The representation defined by
\[\rho(g)(x) = gx\quad (g \in G,\ x \in \mathbb{C}G)\]
is called the {\em left regular representation} of \(G\); the corresponding \(\mathbb{C}G\)-module is called the {\em group module}. This will be useful when we describe irreducible representations of \(S_n\).

It turns out that there are only finitely many irreducible representations of \(G\) over \(\mathbb{C}\), up to isomorpism; any irreducible representation of \(G\) over \(\mathbb{C}\) is isomorphic to a submodule of the group module. Moreover, any representation of \(G\) over \(\mathbb{C}\) is isomorphic to a direct sum of irreducible representations.

If \((\rho,V)\) is a representation of \(V\) over \(\mathbb{C}\), the {\em character} \(\chi_{\rho}\) of \(\rho\) is the map defined by
\begin{eqnarray*}
\chi_{\rho}:  G & \to & \mathbb{C};\\
 \chi_{\rho}(g) &=& \textrm{Tr} (\rho(g)),
\end{eqnarray*}
where \(\textrm{Tr}(\alpha)\) denotes the trace of the linear map \(\alpha\) (i.e. the trace of any matrix of \(\alpha\)). Note that \(\chi_{\rho}(\textrm{Id}) = \dim(\rho)\), and that \(\chi_{\rho}\) is a {\em class function} on \(G\) (meaning that it is constant on each conjugacy class of \(G\).)

The usefulness of characters lies in the following
\begin{fact}
Two complex representations are isomorphic if and only if they have the same character.
\end{fact}

If \(\rho\) is any complex representation of \(G\), its character satisfies \(\chi_{\rho}(g^{-1}) = \overline{\chi_{\rho}(g)}\) for every \(g \in G\). In our case, \(G = S_n\), so \(g^{-1}\) is conjugate to \(g\) for every \(g\), and therefore \(\overline{\chi_{\rho}(g)} = \chi_{\rho}(g)\), so all the characters of \(S_n\) are real-valued.

Given two representations \((\rho,V)\) and \((\rho',V')\) of \(G\), we can form their direct sum, the representation \((\rho \oplus \rho',V \oplus V')\), and their tensor product, the representation \((\rho \otimes \rho',V \otimes V')\). We have \(\chi_{\rho \oplus \rho'} = \chi_{\rho}+ \chi_{\rho'}\), and \(\chi_{\rho \otimes \rho'} = \chi_{\rho} \cdot \chi_{\rho'}\).

Let \(\Gamma\) be a graph on \(G\), and let \(A\) be the adjacency matrix of \(\Gamma\). We may consider \(A\) as a linear operator on \(\mathbb{C}[G]\), acting as follows:
\[A f(g) = \sum_{h\in G:\atop gh \in E(\Gamma)} f(h)\quad \forall f \in \mathbb{C}[G],\ g \in G.\]

We will rely crucially on the following.
\begin{theorem}[Schur; Babai; Diaconis, Shahshahani; Roichman.]
\label{thm:normalcayley}
Let \(G\) be a finite group, let \(S \subset G\) be an inverse-closed, conjugation-invariant subset of \(G\), and let \(\Gamma\) be the normal Cayley graph on \(G\) with generating set \(S\). Let \(A\) denote the adjacency matrix of \(\Gamma\). Let \((\rho_{1},V_{1}),\ldots,(\rho_{k},V_{k})\) be a complete set of non-isomorphic irreducible representations of \(G\) --- i.e., containing one representative from each equivalence class of irreducible representations of \(G\). For each \(i\), let \(U_{i}\) denote the sum of all submodules of the group module \(\mathbb{C}G\) which are isomorphic to \(V_{i}\). We have
\[\mathbb{C}G = \bigoplus_{i=1}^{k}U_{i};\]
for each \(i\), \(U_{i}\) is an eigenspace of \(A\), with dimension \((\dim(V_i))^2\), and eigenvalue
\[\lambda_{V_{i}} = \frac{1}{\dim(V_{i})}\sum_{s \in S} \chi_{\rho_i}(s).\]
\end{theorem}

\begin{remark} Each \(U_i\) is a {\em \(2\)-sided ideal} of \(\mathbb{C}G\), meaning that it is closed under both left and right multiplication by elements of \(\mathbb{C}G\).
\end{remark}
So for any finite group \(G\), \(\mathbb{C}[G]\) can be decomposed into a direct sum of subspaces which are eigenspaces of every normal Cayley graph on \(G\). Hence, the eigenvalues of normal Cayley graphs are relatively easy to handle.

\subsection*{Background on the representation theory of \(S_n\)}
In this section we gather all the necessary background that we need regarding the representation theory of $S_n$. Our treatment follows \cite{sagan}, and is similar to the background given in \cite{jointpaper}, with some additions. Readers familiar with the basics of the theory may wish to omit this section.

A \emph{partition} of \(n\) is a non-increasing sequence of positive integers summing to \(n\), i.e. a sequence $\lambda = (\lambda_1, \ldots, \lambda_k)$ with \(\lambda_{1} \geq \lambda_{2} \geq \ldots \geq \lambda_{k} \geq 1\) and \(\sum_{i=1}^{k} \lambda_{i}=n\); we write \(\lambda \vdash n\). For example, \((3,2,2) \vdash 7\); we sometimes use the shorthand \((3,2,2) = (3,2^{2})\). Let \(\mathcal{P}_n\) denote the set of all partitions of \(n\). The following two orders on \(\mathcal{P}_n\) will be useful.

\begin{definition}
  (Dominance order) Let $\lambda = (\lambda_1, \ldots, \lambda_r)$ and $\mu = (\mu_1,
  \ldots, \mu_s)$ be partitions of $n$. We say that $\lambda \domgeq
  \mu$ ($\lambda$ \emph{dominates} $\mu$) if $\sum_{j=1}^{i} \lambda_i \geq \sum_{j=1}^{i} \mu_i$ for all \(i\) (where we define \(\lambda_{i} = 0 \) for all \(i > r\), and \(\mu_{j} = 0\) for all \(j > s\)).
\end{definition}

It is easy to see that this is a partial order.

\begin{definition}
  (Lexicographic order) Let $\lambda = (\lambda_1, \ldots, \lambda_r)$ and $\mu = (\mu_1,
  \ldots, \mu_s)$ be partitions of $n$. We say that $\lambda > \mu$ if \(\lambda_{j} > \mu_{j}\), where \(j = \min\{i \in [n]:\ \lambda_{i} \neq \mu_{i}\}\).
\end{definition}
It is easy to see that this is a total order which extends the dominance order.

The \emph{cycle-type} of a permutation \(\sigma \in S_{n}\) is the partition of \(n\) obtained by expressing \(\sigma\) as a product of disjoint cycles and listing its cycle-lengths in non-increasing order. The conjugacy classes of \(S_{n}\) are precisely
\[\{\sigma \in S_{n}: \textrm{ cycle-type}(\sigma) = \lambda\}_{\lambda \vdash n}.\]
Moreover, there is an explicit one-to-one correspondence between irreducible representations of \(S_{n}\) (up to isomorphism) and partitions of \(n\), which we now describe.

  Let \(\lambda = (\lambda_{1}, \ldots, \lambda_{k})\) be a partiton of \(n\). The \emph{Young diagram} of $\lambda$ is
  an array of $n$ boxes, or `cells', having $k$ left-justified rows, where row $i$
  contains $\lambda_i$ cells. For example, the Young diagram of the partition \((3,2^{2})\) is:
  \[
\yng(3,2,2)
\]

If the array contains the numbers \(\{1,2,\ldots,n\}\) inside the cells, we call it a \(\lambda\)-\emph{tableau}, or a {\em tableau of shape \(\lambda\)}; for example,
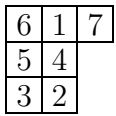
\[
\young(617,54,32)
\]
is a \((3,2^{2})\)-tableau. Two \(\lambda\)-tableaux are said to be \emph{row-equivalent} if they have the same numbers in each row; if a \(\lambda\)-tableau \(T\) has rows \(R_{1},\ldots,R_{l} \subset [n]\) and columns \(C_{1},\ldots,C_{l} \subset [n]\), then we let \(R_{T} = S_{R_{1}} \times S_{R_{2}} \times \ldots \times S_{R_{l}}\) be the row-stablizer of \(T\) and \(C_T = S_{C_{1}} \times S_{C_{2}} \times \ldots \times S_{C_{k}}\) be the column-stabilizer.

A \(\lambda\)-\emph{tabloid} is a \(\lambda\)-tableau with unordered row entries (or formally, a row-equivalence class of \(\lambda\)-tableaux); given a tableau \(T\), we write \([T]\) for the tabloid it produces. For example, the \((3,2^{2})\)-tableau above produces the following \((3,2^{2})\)-tabloid:\\
\\
\begin{tabular}{ccc} \{1 & 6 & 7\}\\
\{4 &\ 5\} &\\
\{2 &\ 3\} &\\
\end{tabular}\\
\\
\\
Consider the natural left action of \(S_{n}\) on the set \(X^{\lambda}\) of all \(\lambda\)-tabloids; let \(M^{\lambda} = \mathbb{C}[X^{\lambda}]\) be the corresponding permutation module, i.e. the complex vector space with basis \(X^{\lambda}\) and \(S_{n}\) action given by extending linearly. Each \(M^{\lambda}\) has a special irreducible submodule, which we now describe.
\begin{definition}
Let \(\lambda\) be a partition of \(n\). If $T$ is a \(\lambda\)-tableau, let 
$$\kappa_T := \sum_{\pi \in C_T} \sgn(\pi) \pi\quad \in \mathbb{C}S_n,$$
and let $e_T := \kappa_T \{ T \} \in M^{\lambda}$. The {\em Specht module} $S^{\lambda}$ is the submodule of $M^{\lambda}$ spanned by
$$\{e_T:\ T \textrm{ is a } \lambda\textrm{-tabloid}\}.$$
\end{definition}

\begin{theorem}
The Specht modules \(\{S^{\lambda}:\ \lambda \vdash n\}\) are a complete set of pairwise inequivalent, irreducible representations of $S_n$.
\end{theorem}

Hence, any irreducible representation \(\rho\) of \(S_{n}\) is isomorphic to some \(S^{\lambda}\). For example, \(S^{(n)} = M^{(n)}\) is the trivial representation; \(M^{(1^{n})}\) is the left-regular representation (corresponding to the group module \(\mathbb{C}S_{n}\)), and \(S^{(1^{n})}\) is the sign representation.

From now on we will write \([\lambda]\) for the equivalence class of the irreducible representation \(S^{\lambda}\), \(\chi_{\lambda}\) for the irreducible character \(\chi_{S^{\lambda}}\), and \(\xi_{\lambda}\) for the permutation character \(\chi_{M^{\lambda}}\).

If \(U \in [\alpha],\ V \in [\beta]\), we define \([\alpha]+[\beta]\) to be the equivalence class of \(U\oplus V\), and \([\alpha] \otimes [\beta]\) to be the equivalence class of \(U \otimes V\); since \(\chi_{U \oplus V} = \chi_{U}+\chi_{V}\) and \(\chi_{U \otimes V} = \chi_{U} \cdot \chi_{V}\), this corresponds to pointwise addition/multiplication of the corresponding characters.

We will need the following well-known facts about the permutation characters.

\begin{lemma} \label{lemma:dominance-lemma}
  Let $\lambda$ be a partition of $n$, and let $\sigma \in S_n$. If
  $\xi_{\lambda}(\sigma) \neq 0$, then $\textrm{cycle-type}(\sigma) \domleq \lambda$. Moreover, if \(\textrm{cycle-type}(\sigma) = \lambda\), then \(\xi_{\lambda}(\sigma) = 1\).
\end{lemma}

\begin{proof}
  The set of $\lambda$-tabloids is a basis for the permutation module $M^{\lambda}$. Thus, $\xi_{\lambda}(\sigma)$, which is the trace of the corresponding representation on the permutation \(\sigma\), is simply the number of \(\lambda\)-tabloids fixed by \(\sigma\). If $\xi_{\lambda}(\sigma) \neq 0$, then $\sigma$ fixes some \(\lambda\)-tabloid \([T]\). Hence, every row of length $l$
  in \([T]\) is a union of the sets of numbers in a collection of disjoint cycles of total
  length $l$ in $\sigma$. Thus, the cycle-type of \(\sigma\) is a refinement of \(\lambda\), and therefore $\textrm{cycle-type}(\sigma) \domleq \lambda$, as required. If \(\sigma\) has cycle-type \(\lambda\), then it fixes just one \(\lambda\)-tabloid, the one whose rows correspond to the cycles of \(\sigma\), so \(\xi_{\lambda}(\sigma) = 1\). 
\end{proof}

Since the lexicographic order extends the dominance order, we immediately obtain the following.

\begin{corollary}
\label{corr:permuppertriangular}
Let $N$ denote the permutation-character table of $S_n$, with rows and columns
  indexed by partitions of \(n\) in decreasing lexicographic
  order. (So $N_{\lambda, \mu} = \xi_{\lambda}(\sigma_{\mu})$, where
  $\sigma_{\mu}$ is a permutation with cycle-type $\mu$.) Then \(N\) is {\em strictly upper-triangular} (meaning that it is upper-triangular with 1's all along the main diagonal).
\end{corollary}

We now explain how the permutation modules \(M^{\mu}\) decompose into irreducibles.

\begin{definition}
Let \(\lambda,\mu\) be partitions of \(n\). A \emph{generalized} \(\lambda\)-\emph{tableau} is produced by placing a number between 1 and \(n\) in each space of the Young diagram of \(\lambda\); if it has \(\mu_{i}\) \(i\)'s \((1 \leq i \leq n)\), then it is said to have \emph{content} \(\mu\). A generalized \(\lambda\)-tableau is said to be \emph{semistandard} if the numbers are non-decreasing along each row and strictly increasing down each column.
\end{definition}

  \begin{definition} \label{def:kostka}
  Let $\lambda, \mu$ be partitions of $n$. The {\em Kostka number}
  $K_{\lambda,\mu}$ is the number of semistandard generalized $\lambda$-tableaux
  with content $\mu$.
\end{definition}

\begin{theorem}[Young's Rule]
\label{thm:young-rule}
Let $\mu$ be a partition of $n$. Then the permutation module
  $M^{\mu}$ decomposes as
  \[M^{\mu} \cong \bigoplus_{\lambda \vdash n} K_{\lambda, \mu}
  S^{\lambda}.\]
  Hence,
  \[\xi_{\mu}=\sum_{\lambda \vdash n} K_{\lambda,\mu}\chi_{\lambda}.\]
\end{theorem}

We will need the following well-known facts about the Kostka numbers.

\begin{lemma}
 \label{lemma:kostka}
Let \(\lambda\) and \(\mu\) be partitions of \(n\). If \(K_{\lambda,\mu} \geq 1\), then \(\lambda \domgeq \mu\). Moreover, \(K_{\lambda,\lambda}=1\) for all \(\lambda\).
\end{lemma}
\begin{proof}
If \(K_{\lambda,\mu} \geq 1\), then there exists a semistandard generalized \(\lambda\)-tableau of content \(\mu\). All \(\mu_{i}\) \(i\)'s must appear in the first \(i\) rows of this tableau, so \(\sum_{j=0}^{i} \mu_{j} \leq \sum_{j=0}^{i} \lambda_{j}\) for each \(i\). Hence, \(\lambda \domgeq \mu\). Observe that the generalized \(\lambda\)-tableau with \(\lambda_{i}\) \(i\)'s in the \(i\)th row is the unique semistandard generalized \(\lambda\)-tableau of content \(\lambda\), so \(K_{\lambda,\lambda}=1\) for every partition \(\lambda\).
\end{proof}
Again, since the lexicographic order extends the dominance order, we immediately obtain the following.

\begin{corollary} \label{corr:uppertriangular}
Let \(K\) denote the Kostka matrix, with rows and columns
  indexed by partitions of \(n\) in decreasing lexicographic
  order. Then \(K\) is strictly upper-triangular, so for any partition \(\gamma\) of \(n\),
\[\Span\{\chi_{\alpha}:\ \alpha \geq \gamma\} = \Span\{\xi_{\beta}:\ \beta \geq \gamma\}.\]
\end{corollary}

On the other hand, we can express the irreducible characters in terms of the permutation characters using the {\em determinantal formula}: for any partition \(\alpha\) of \(n\),
\begin{equation}\label{eq:determinantalformula} \chi_{\alpha} = \sum_{\pi \in S_{n}} \sgn(\pi) \xi_{\alpha - \textrm{id}+\pi}.\end{equation}
Here, if \(\alpha = (\alpha_{1},\alpha_{2},\ldots,\alpha_{l})\), \(\alpha - \textrm{id}+\pi\) is defined to be the sequence \((\alpha_{1}-1+\pi(1),\alpha_{2}-2+\pi(2),\ldots,\alpha_{l}-l+\pi(l))\). If this sequence has all its entries non-negative, we let \(\overline{\alpha-\textrm{id}+\pi}\) be the partition of \(n\) obtained by reordering its entries, and we define \(\xi_{\alpha - \textrm{id}+\pi} = \xi_{\overline{\alpha-\textrm{id}+\pi}}\). If the sequence has a negative entry, we define \(\xi_{\alpha - \textrm{id}+\pi} = 0\). Note that if \(\xi_{\beta}\) appears on the right-hand side of (\ref{eq:determinantalformula}), then \(\beta \geq \alpha\), so the determinantal formula expresses \(\chi_{\alpha}\) in terms of \(\{\xi_{\beta}: \ \beta \geq \alpha\}\).

The restriction of an irreducible representation of \(S_{n}\) to the subgroup \(\{\sigma \in S_{n}: \sigma(i)=i \ \forall i > n-k\} = S_{n-k}\) can be decomposed into irreducible representations of \(S_{n-k}\) as follows.

\begin{theorem}[The Branching Rule]
\label{thm:branching-rule}
   Let \(\alpha\) be a partition of \(n-k\), and \(\lambda\) a partition of \(n\). We write $\alpha \subset^k \lambda$ if the Young diagram of $\alpha$ can be produced from that of \(\lambda\) by sequentially removing \(k\) corners (so that after removing the \(i\)th corner, we have the Young diagram of a partition of \(n-i\).) Let \(a_{\alpha,\lambda}\) be the number of ways of doing this; then we have
  \[[\lambda] \downarrow S_{n-k} = \sum_{\alpha \vdash n-k:\ \alpha \subset^{k} \lambda} a_{\alpha,\lambda}[\alpha],\]
  and therefore
  \[\chi_{\lambda} \downarrow S_{n-k} = \sum_{\alpha \vdash n-k:\ \alpha \subset^{k} \lambda} a_{\alpha, \lambda}\chi_{\alpha}.\]
  \end{theorem}

\begin{definition}
Let \(\lambda = (\lambda_{1},\ldots,\lambda_{k})\) be a partition of \(n\); if its Young diagram has columns of lengths \(\lambda_{1}' \geq \lambda_{2}' \geq \ldots \geq \lambda_{l}' \geq 1\), let \(\lambda^{\top}\) denote the partition \((\lambda_{1}'\ldots,\lambda_{l}')\) of \(n\). The partition \(\lambda^\top\) is called the {\em transpose} of \(\lambda\). (Its Young diagram is the transpose of that of \(\lambda\)).
\end{definition}

There is a useful relationship between the Specht modules \(S^{\lambda}\) and \(S^{\lambda^{\top}}\):

\begin{theorem}
\label{thm:transpose-sign}
If \(\lambda\) is a partition of \([n]\), then \(S^{(\lambda^{\top})} \cong S^{\lambda} \otimes S^{(1^n)}\). Hence, \(\chi_{\lambda^{\top}} = \chi_{\lambda} \cdot \textrm{sgn}\).
\end{theorem}

If \(\lambda\) is a partition of \(n\), we write \(f^{\lambda} = \dim[\alpha]\) for the dimension of the corresponding irreducible representation. There is a surprisingly simple formula for \(f^{\lambda}\), in terms of the {\em hook-lengths} of the Young diagram of \(\lambda\).

\begin{definition}
  The {\em hook} of a cell $(i,j)$ in the Young diagram of a partition $\lambda$
  is $H_{i,j} = \{ (i,j') : j' \geq j \} \cup \{ (i',j) : i' \geq i \}$, i.e. the set of cells consisting of \((i,j)\), all cells to the right in the same row and all cells below in the same column. The {\em hook-length} of $(i,j)$ is $h_{i,j} = |H_{i,j}|$.
\end{definition}

\begin{theorem}[The Hook formula, \cite{frt}]
\label{hook-formula}
If \(\lambda\) is a partition of \(n\) with hook-lengths $h_1, \ldots, h_k$, then
\begin{equation}\label{eq:hook}f^{\lambda} = \frac{n!}{\prod_i h_i}.\end{equation}
\end{theorem}

\section{Preliminary results}
In order to construct a pseudoadjacency matrix of the \(t\)-derangment graph with the right eigenvalues, we will need several preliminary results on the representation theory of \(S_n\).

\subsection*{The Kostka matrix and the permutation character table}
If \(k \in \mathbb{N}\), we let
\[\mathcal{F}_{n,k} =  \{\lambda \vdash n:\ \lambda_1 \geq n-k\}\]
denote the set of all partitions of \(n\) whose Young diagram has first row of length at least \(n-k\). Note that
\[\mathcal{F}_{n,k} = \{\lambda \vdash n:\ \lambda \geq (n-k,1^k)\}.\]
We will need the following fact about the minors of the Kostka matrix and the permutation character table with rows and columns indexed by \(\mathcal{F}_{n,k}\).

\begin{proposition}
\label{prop:independent}
If \(k \in \mathbb{N}\), let \(\tilde{K}\) denote the top-left minor of the Kostka matrix \(K\) with rows and columns indexed by \(\mathcal{F}_{n,k}\), and let \(\tilde{N}\) denote the top-left minor of the permutation character table \(N\) with rows and columns indexed by \(\mathcal{F}_{n,k}\). Then \(\tilde{K}\) and \(\tilde{N}\) do not depend on \(n\), provided \(n > 2k\).
\end{proposition}
\begin{proof}
Let $\lambda \geq
  (n-k,1^k)$ be a partition. Then $\lambda_1 \geq n-k$.  Write
  $\lambda' = (\lambda_1 - (n-k), \lambda_2, \ldots)$. (Note that this may not be a {\em bona fide} partition, as it may not be in non-increasing order.) Now the
  mapping $\lambda \mapsto \lambda'$ has the same image over $\{
  \lambda : \lambda \geq (n-k,1^k) \}$ for all $n \geq 2k$:
  namely, `partitions' of $k$ where the first row is not necessarily the
  longest.

We first consider \(K\). Recall that $K_{\lambda,\mu}$ is the number of semistandard \(\lambda\)-tableaux of content $\mu$. Let $T$ be a semistandard
  \(\lambda\)-tableau of content $\mu$; we now count the number of choices for \(T\). Since the numbers in a semistandard
  tableau are strictly increasing down each column and non-increasing along each row, and \(\mu_1 \geq n-k\), we must always place 1's in the first \(n-k\) cells of the first row of \(T\). We must now fill the rest of the cells with content \(\mu'\). Provided \(n \geq 2k\), \(\mu'\) is independent of \(n\), and the remaining cells in the first row have no cells below them, so the number of ways of doing this is independent of \(n\). Hence, the entire minor \(\tilde{K}\) is independent of $n$.

  Now consider $N$. Recall that \(N_{\lambda,\mu} = \xi_{\lambda}(\sigma_{\mu})\) is the number of \(\lambda\)-tabloids fixed by $\sigma_{\mu}$. To count these, first note that the numbers in the long cycle
  of $\sigma$ (which has length at least \(n-k\)) must all be in the first row of the \(\lambda\)-tabloid (otherwise
  the long cycle of $\sigma$ must intersect two or more rows, as \(n-k > k\).)
  This leaves us with a $(\lambda_1 - \mu_1, \lambda_2,
  \ldots, \lambda_r)$-`tableau', which we need to fill with the remaining $n -
  \mu_1$ elements in such a way that $\sigma$ fixes it.  It is
  easy to see that, again, the number of ways of doing this is independent of \(n\).
\end{proof}

\begin{remark}
By Corollaries \ref{corr:permuppertriangular} and \ref{corr:uppertriangular}, \(\tilde{K}\) and \(\tilde{N}\) are strictly upper-triangular, and are therefore invertible.
\end{remark}

\begin{remark}
If \(n \geq 2k\), then \(|\mathcal{F}_{n,k}|\) is independent of \(n\); we denote it by \(q_k\). Note that
\[q_k = \sum_{s=0}^{k}p_s,\]
where \(p_s\) denotes the number of partitions of \(s\) (and \(p_0 :=1\)).
\end{remark}

\subsection*{The critical subspace of \(\mathbb{C}S_n\)}
Let \(\alpha\) be a partition of \(n\). Using the notation of Theorem \ref{thm:normalcayley}, \(U_{S^{\alpha}}\) denotes the sum of all submodules of \(\mathbb{C}S_n\) which are isomorphic to \(S^{\alpha}\). Recall that \(U_{S^\alpha}\) is a 2-sided ideal of \(\mathbb{C}S_n\). For brevity, we write \(U_{\alpha}\) in place of \(U_{S^{\alpha}}\), and refer to it as the {\em \(\alpha\)-isotypical ideal}. Observe that \(U_{(n)}\) is the subspace of constant vectors in \(\mathbb{C}S_n\).

Note that the permutation action of \(S_n\) on \(t\)-sets is the same as its action on \((n-t,t)\)-tabloids; the corresponding permutation representation is \(M^{(n-t,t)}\). Not surpringly, this representation and its irreducible constituents play a crucial role in our proof. We have the following
\begin{lemma}
\label{lemma:decomp}
For any positive integers \(t \leq n\),
\begin{equation} \label{eq:tsetdecomposition} M^{(n-t,t)} \cong \bigoplus_{s=0}^{t} S^{(n-s,s)}.\end{equation}
\end{lemma}
\begin{proof}
By Lemma \ref{lemma:kostka}, if \(K_{\lambda,(n-t,t)} \geq 1\), then \(\lambda \domgeq (n-t,t)\), so \(\lambda = (n-s,s)\) for some \(s \in \{0,1,2,\ldots,t\}\). For any \(s \in \{0,1,2,\ldots,t\}\), we have \(K_{(n-s,s),(n-t,t)} = 1\), since there is a unique semistandard generalized \((n-s,s)\)-tableau of content \((n-t,t)\) --- namely, the tableau with \(1\)'s in the first \(n-t\) cells of the first row, and \(2\)'s in all the other cells. Hence, (\ref{eq:tsetdecomposition}) follows from Theorem \ref{thm:young-rule}.
\end{proof}

\begin{corollary}
\label{corr:dims}
For any positive integers \(t \leq n\), we have
\[\xi_{(n-t,t)} = \sum_{s=0}^{t} \chi_{(n-s,s)}.\]
Hence,
\[\chi_{(n-t,t)} = \xi_{(n-t,t)}-\xi_{(n-t+1,t-1)},\]
so in particular,
\[f^{(n-t,t)} = {n \choose t}-{n \choose t-1}.\]
\end{corollary}

We call a partition of \(n\) {\em critical} if it is of the form \((n-s,s)\) for some \(s \in \{0,1,\ldots,t\}\), i.e. if the corresponding irreducible representation is a constituent of the permutation module \(M^{(n-t,t)}\). Let \(\mathcal{C}_{n,t}\) denote the set of all critical partitions of \(n\), and let
\[U_t = \bigoplus_{s=0}^{t} U_{(n-s,s)}\]
denote the direct sum of the corresponding isotypical ideals; we call \(U_t\) the {\em critical subspace} of \(\mathbb{C}S_n\). Our aim in this section is to show that \(U_t\) is precisely the subspace spanned by the characteristic functions of the \(t\)-SSC's (the cosets of stabilizers of \(t\)-sets). For this, we will need the following well-known fact from general representation theory.
\begin{lemma}
\label{lemma:righttranslation}
If $G$ is a finite group, and $M,M'$ are two isomorphic submodules of $\mathbb{C}G$, then there exists $s \in \mathbb{C}G$ such that the right multiplication map $x \mapsto xs$ is an isomorphism from $M$ to $M'$.
\end{lemma}
For a proof, see for example \cite{jamesliebeck}.

We can now prove the main result of this section.
\begin{theorem}
\label{thm:cosets-span}
For any \(n\in \mathbb{N}\) and any \(t \leq n\), we have
\[U_{t} = \Span\{1_{T_{x \mapsto y}}:\ x,y \in [n]^{(t)}\}.\]
\end{theorem}
\begin{proof}
Observe that the set of vectors
\[\left\{\sum_{\sigma \in S_n:\atop \sigma([t])=y} \sigma\quad :\ y \in [n]^{(t)}\right\}\]
is a basis for a copy $W$ of the permutation module $M^{(n-t,t)}$ in the group module $\mathbb{C}S_{n}$. Let $V$ be the sum of all right translates of $W$, i.e. the subspace of $\mathbb{C}S_{n}$ spanned by all the \(1_{T_{x \mapsto y}}\)'s. We wish to prove that $V=U_t$.

By Lemma \ref{lemma:decomp}, we have $W \leq U_{t}$. Since $U_{t}$ is closed under right-multiplication by elements of $\mathbb{C}S_{n}$, we have $V \leq U_{t}$ as well.

It remains to show that \(U_t \leq V\). By Lemma \ref{lemma:righttranslation}, the sum of all right translates of $W$ contains all submodules of $\mathbb{C}S_{n}$ isomorphic to $S^{(n-s,s)}$, for each \(s \in \{0,1,\ldots,t\}\). In other words, $U_{(n-s,s)} \leq V\) for each \(s \in \{0,1,\ldots,t\}\). Hence, $U_{t} \leq V$, so $V=U_{t}$, as required.
\end{proof}

\subsection*{Lower bounds on the dimensions of Specht modules}
We will also need various lower bounds on the dimensions of Specht modules. The following result was proved in \cite{jointpaper}; we repeat the proof for the reader's convenience.
\begin{lemma} \label{lemma:dimension-for-long}
  Let $\alpha$ be a partition of \(n\) with \(\alpha_1=n-t\). Then
  $$f^{\alpha} > e^{-t} {n \choose t}.$$
\end{lemma}
\begin{proof}
By the Hook formula, (\ref{eq:hook}), it suffices to prove that the product of the hook-lengths of \(\alpha\) satisfies
\[\prod_{i,j} h_{i,j} < t!(n-t)! e^{t}.\]
Delete the first row \(R_1\) of the Young diagram of \(\alpha\); the resulting Young diagram \(D\) corresponds to a partition of \(t\), and therefore a representation of \(S_t\), which has dimension
\[\frac{t!}{\prod_{(i,j) \in D} h_{i,j}} \geq 1.\]
Hence,
\[\prod_{(i,j) \in D} h_{i,j} \leq t!.\]
We now bound the product of the hook-lengths of the cells in the first row; this is of the form
\[\prod_{(i,j) \in R_1} h_{i,j} = \prod_{j=1}^{n-t}(j+c_j),\]
where \(\sum_{j=1}^{n-t}c_j = t\). Using the AM/GM inequality, we obtain:
\begin{eqnarray*}\prod_{j=1}^{n-t} \frac{j+c_j}{j}
    & = &
    \prod_{j=1}^{n-t} \left( 1 + \frac{c_j}{j} \right) \\
     & \leq &
    \left( \sum_{j=1}^{n-t} \frac{1 + \frac{c_j}{j}}{n-t} \right)^{n-t} \\
    & \leq &
    \left( \frac{n-t + \sum_{j=1}^{n-t} c_j}{n-t} \right)^{n-t} \\
    & = &
    \left( 1 + \frac{t}{n-t} \right)^{n-t}\\
 & < & e^t.
  \end{eqnarray*}
Hence,
\[\prod_{i,j} h_{i,j} < t!(n-t)!e^{t},\]
as desired.
\end{proof}
Let \(t \in \mathbb{N}\) be fixed. As above, we let
\[\mathcal{F}_{n,t} =  \{\alpha \vdash n:\ \alpha_1 \geq n-t\}\]
denote the set of all partitions of \(n\) whose Young diagram has first row of length at least \(n-t\). We call these the {\em fat} partitions of \(n\). Note that
\[\mathcal{F}_{n,t} = \{\alpha \vdash n:\ \alpha \geq (n-t,1^t)\}.\]
Let
\[\mathcal{F}_{n,t}' = \{\alpha \vdash n:\ \alpha_1' \geq n-t\}\]
denote the set of all partitions of \(n\) whose Young diagram has first column of depth at least \(n-t\). We call these the {\em tall} partitions of \(n\). Note that
\[\mathcal{F}_{n,t}' = \{\alpha^{\top}:\ \alpha \in \mathcal{F}_{n,t}\}:\]
the tall partitions are simply the transposes of the fat partitions. We call a partition {\em medium} if it is neither fat nor tall; we let
\[\mathcal{M}_{n,t} = \mathcal{P}_n \setminus (\mathcal{F}_{n,t} \cup \mathcal{F}_{n,t}') = \{\alpha \in \mathcal{P}_n:\ \alpha_1,\alpha_1' < n-t\}\]
denote the set of medium partitions.

We will use different arguments to analyse the eigenvalues corresponding to the fat, tall and medium partitions. The following lower bound on the dimensions of the medium Specht modules will be useful.

\begin{lemma}
\label{lemma:highdimension}
Let \(t \in \mathbb{N}\) be fixed. Then there exists \(c_t>0\) such that for all medium partitions \(\alpha\) of \(n\),
\[f^{\alpha} \geq c_t n^{t+1}.\]
\end{lemma}
\begin{proof}
Let \(t \in \mathbb{N}\) be fixed. We proceed by induction on \(n\). By choosing \(c_t >0\) sufficiently small, we may assume that the statement of the lemma holds for all \(n\) such that \((n-t-1)(n-t-2) \ldots (n-2t-1) < \tfrac{1}{2} n^{t+1}\).

Let \(n\) be such that
\begin{equation}\label{eq:assumptiononn}(n-t-1)(n-t-2) \ldots (n-2t-1) \geq \tfrac{1}{2} n^{t+1}\end{equation}
Assume that the statement of the lemma holds for \(n-2\) and \(n-1\); we will prove it for \(n\). Let \(\alpha\) be an partition of \(n\) with \(\dim[\alpha]\ (\ = f^{\alpha}) < c_{t} n^{t+1}\). We wish to show that \(\alpha \in \mathcal{F}_{n,t} \cup \mathcal{F}_{n,t}'\). Consider the restriction \([\alpha]\downarrow S_{n-1}\), which has the same dimension as \([\alpha]\).

First suppose that \([\alpha]\downarrow S_{n-1}\) is reducible. Suppose it has \([\beta]\) as a constituent for some \(\beta \in \mathcal{F}_{n-1,t} \cup \mathcal{F}_{n-1,t}'\). It follows from the Branching Rule that \(\alpha_1 = \beta_1\) or \(\beta_1+1\). If \(\beta_{1} \geq n-t\), then \(\alpha_{1} \geq n-t\), so \(\alpha \in \mathcal{F}_{n,t}\). Similarly, if \(\beta_{1}' \geq n-t\), then \(\alpha_{1}' \geq n-t\), so \(\alpha \in \mathcal{F}_{n,t}'\).

Suppose that \(\beta_{1} = n-t-1\). If \(\alpha_1 = n-t\), then \(\alpha \in \mathcal{F}_{n,t}\), so we are done. Hence, we may assume that \(\alpha_{1} = n-t-1\). We proceed to bound \(\dim [\alpha]\) from below using the Hook formula. Notice that for \(j \geq t+2\), \(\alpha_{j}' \leq 1\), so the hook lengths of \(\alpha\) satisfy \(h^{\alpha}_{1,j} = n-t-j\); for \(1 \leq j \leq t+1\) we trivially have \(h^{\alpha}_{1,j} \leq n+1-j\). Also, since there are just \(t+1\) spaces below the first row of the Young diagram of \(\alpha\), we have \(\prod_{i \geq 2, j \geq 1} h^{\alpha}_{i,j} \leq (t+1)!\). Hence, the product of the hook lengths satisfies:
\[\prod_{i,j}h^{\alpha}_{i,j} \leq n(n-1)\ldots(n-t) (n-2t-2)! (t+1)!\]
and therefore, by the Hook formula,
\[\dim [\alpha] \geq \frac{(n-t-1)(n-t-2) \ldots (n-2t-1)}{(t+1)!} \geq \frac{1}{2(t+1)!} n^{t+1} \geq c_t n^{t+1},\]
provided we choose \(c_t \leq \tfrac{1}{2(t+1)!}\). By symmetry, the same conclusion holds if \(\beta_{1}' = n-t-1\) and \(\alpha_{1}' = n-t-1\).

Hence, we may assume that each irreducible constituent \([\beta]\) of \([\alpha] \downarrow S_{n-1}\) has \(\beta_1,\beta_1' < n-t-1\). Hence, by the induction hypothesis for \(n-1\), each constituent has dimension \(\geq c_{t} (n-1)^{t+1}\). By our assumption (\ref{eq:assumptiononn}) on \(n\), \(2c_{t}(n-1)^{t+1} \geq c_{t}n^{t+1}\), so there is just one constituent, i.e. \([\alpha] \downarrow S_{n-1}\) is irreducible. The only way this can happen is if \(\alpha\) has rectangular Young diagram, i.e. \(\alpha = (a^{b})\) for some \(a,b \in \mathbb{N}\) such that \(ab=n\). Since \(b \geq 2\), we have \(a \leq n/2 < n-2-t\) (since \(n > 2t+4\), by our assumption (\ref{eq:assumptiononn}) on \(n\)); similarly, \(b < n-2-t\).

Now consider the restriction
\[[\alpha]\downarrow S_{n-2} = [a^{b-1},a-2] + [a^{b-2},a-1,a-1]\]
Note that both of these irreducible constituents have Young diagram with first row of length \(\leq a < n-2-t\) and first column of length \(\leq b < n-2-t\), and therefore by the induction hypothesis for \(n-2\), have dimension \(\geq c_{t} (n-2)^{t+1}\). By our assumption (\ref{eq:assumptiononn}) on \(n\), we have \(2c_{t}(n-2)^{t+1} \geq c_{t}n^{t+1}\), contradicting \(\dim ([\alpha] \downarrow S_{n-2}) < c_{t}n^{t+1}\). This completes the proof.
\end{proof}

\subsection*{The number of even/odd permutations in \(S_n\) with no short cycles}
Let \(D_{n,t}\) denote the number of permutations in \(S_{n}\) with no cycle of length at most \(t\), and let \(E_{n,t}\) and \(O_{n,t}\) denote the number of these permutations which are respectively even and odd. We will need the following.
\begin{lemma}
\label{lemma:noshortcyclesbound}
For each \(t \in \mathbb{N}\), there exists \(L_{t} >0\) such that \(\min\{E_{n,t},O_{n,t}\} \geq L_{t} n!\) for all \(n \geq 2t+2\).
\end{lemma}
\begin{proof}
First suppose that \(2t+2 \leq n \leq 3t+2\). If \(n\) is odd, then the even permutations with no cycles of length \(\leq t\) are precisely the \(n\)-cycles, and the odd ones are precisely the permutations with exactly two cycles, both of length \(\geq t+1\); if \(n\) is even, the situation is reversed. Choose \(L_{t} > 0\) such that \(\min\{E_{n,t},O_{n,t}\} \geq L_{t} n!\) whenever \(2t+2 \leq n \leq 3t+2\); it is easy to check that we can take \(L_{t} = \frac{2}{(3t+2)^{2}}\).

We now derive recurrence relations for \(E_{n,t}\) and \(O_{n,t}\). Let \(\sigma\) be an even permutation with no cycle of length \(\leq t\). Let \(i=\sigma(n)\); then we may write \(\sigma = (ni) \rho\) where \(\rho\) is an odd permutation of \([n-1]\) and has no cycle of length \(\leq t\) except possibly a \(t\)-cycle containing \(i\). Conversely, given any such pair \(\rho,i\), \((ni)\rho\) has no cycle of length \(\leq t\). Hence, we have
\begin{eqnarray*}
E_{n,t} & = & (n-1)(O_{n-1,t}+ (n-2)(n-3)\ldots(n-t)E_{n-t-1,t})\quad \textrm{if }t \textrm{ is even};\\
E_{n,t}& = &(n-1)(O_{n-1,t}+ (n-2)(n-3)\ldots(n-t)O_{n-t-1,t})\quad \textrm{if }t \textrm{ is odd}.
\end{eqnarray*}
Similarly,
\begin{eqnarray*}
O_{n,t} & = & (n-1)(E_{n-1,t}+ (n-2)(n-3)\ldots(n-t)O_{n-t-1,t})\quad \textrm{if }t \textrm{ is even};\\
O_{n,t}& = &(n-1)(E_{n-1,t}+ (n-2)(n-3)\ldots(n-t)E_{n-t-1,t})\quad \textrm{if }t \textrm{ is odd}.
\end{eqnarray*}
We can now prove the lemma by induction on \(n\). Let \(n \geq 3t+3\) and assume the statement is true for smaller values of \(n\). The recurrence relations above give
\[E_{n,t},O_{n,t} \geq (n-1)(L_{t}(n-1)!+(n-2)(n-3)\ldots(n-t)L_{t}(n-t-1)!) = L_{t}n!,\]
as required.
\end{proof}

\section{Construction of the pseudoadjacency matrix}
Our aim is now to construct a pseudoadjacency matrix \(A\) with the right eigenvalues to deduce Theorem \ref{thm:main} from Theorem \ref{thm:pseudo}. Our pseudoadjacency matrix will be a linear combination of adjacency matrices of Cayley graphs generated by conjugacy classes of \(t\)-derangements in \(S_n\).

Note that if \(\Cay(S_n,X)\) is any normal Cayley graph on \(S_n\), then by Theorem \ref{thm:normalcayley}, for each partition \(\alpha\) of \(n\), the subspace \(U_{\alpha}\) is an eigenspace of \(\Cay(S_n,X)\), and the corresponding eigenvalue is
\[\lambda_{\alpha} := \frac{1}{f^\alpha} \sum_{\sigma \in X} \chi_{\alpha}(\sigma).\]

In particular, if \(\mu\) is a partition of \(n\), and \(X_{\mu}\) denotes the conjugacy class of permutations in \(S_n\) with cycle-type \(\mu\), then the eigenvalues of the normal Cayley graph \(\Cay(S_n,X_{\mu})\) are given by
\[\lambda^{\mu}_\alpha = \frac{1}{f^\alpha} \sum_{\sigma \in X_{\mu}} \chi_{\alpha}(\sigma)\quad (\alpha \vdash n).\]

Let \(A^{\mu}\) denote the adjacency matrix of \(\Cay(S_n,X_{\mu})\). Let \(A\) be a real linear combination of the \(A^{\mu}\)'s, i.e.
\[A^{(w)} = \sum_{\mu \vdash n} w_{\mu} A^{\mu},\]
where each \(w_{\mu} \in \mathbb{R}\). (Surprisingly, it will be necessary for us to allow \(w_{\mu}<0\).) We may regard \(w\) as a class function on \(S_n\): for each \(\sigma \in S_n\), simply define \(w(\sigma) = w_{\mu}\), where \(\mu\) is the cycle-type of \(\sigma\). Note that
\[A^{(w)}_{\sigma,\pi} = w(\sigma^{-1} \pi) \quad \forall \sigma,\pi \in S_n.\]

Each \(U_{\alpha}\) is again an eigenspace of \(A^{(w)}\); the corresponding eigenvalue is clearly
\[\lambda^{(w)}_\alpha = \frac{1}{f^\alpha} \sum_{\sigma \in S_n} w(\sigma) \chi_{\alpha}(\sigma).\]
In terms of the inner product
\[\langle \phi, \psi \rangle = \sum_{\sigma \in S_n}\phi(\sigma) \overline{\psi(\sigma)}\]
on \(\mathbb{C}[S_n]\), this becomes
\[\lambda^{(w)}_\alpha = \frac{1}{f^\alpha} \langle w, \chi_{\alpha} \rangle.\]

Since \(U_{(n)}\) is the subspace of constant functions in \(\mathbb{C}[S_n]\), and \(\chi_{(n)} \equiv 1\), we have
\[\lambda^{(w)}_{\const} = \lambda^{(w)}_{(n)} = \langle w, \chi_{(n)} \rangle = \sum_{\sigma \in S_n} w(\sigma).\]

Provided \(w\) is supported on \(\mathcal{D}_{(t)}\), the set of all \(t\)-derangments of \(S_n\), the matrix \(A^{(w)}\) is a pseudoadjacency matrix for \(\Gamma_{(t)}\). Our aim is to construct (for \(n\) sufficiently large depending on \(t\)), a class function \(w\) supported on the \(t\)-derangments, with
\begin{equation}\label{eq:quotient}\lambda^{(w)}_{\min}/\lambda^{(w)}_{\const} = -\frac{1}{{n \choose t}-1} := \nu_{n,t}.\end{equation}
The upper bound in Theorem \ref{thm:main} will then follow from Theorem \ref{thm:pseudo}.

Since rescaling \(w\) makes no difference to (\ref{eq:quotient}), we will construct a class function \(w\) with \(\lambda^{(w)}_{\const}=1\) and \(\lambda^{(w)}_{\min} = \nu_{n,t}\). Since equality has to hold in Theorem \ref{thm:pseudo} when the independent set is a coset of the stabilizer of a \(t\)-set, i.e. of the form \(T_{x \mapsto y}\) for some \(x,y \in [n]^{(t)}\), we know that for any \(x,y \in [n]^{(t)}\), we must have
\[1_{T_{x \mapsto y}} - \frac{1}{{n \choose t}}\mathbf{f} \in \Ker(\nu_{n,t} I - A).\]
By Theorem \ref{thm:cosets-span}, it follows that we must have \(U_{(n-s,s)} \leq \Ker(\nu_{n,t} I-A)\) for all \(s \in [t]\), i.e.
\[\lambda^{(w)}_{(n-s,s)} = \nu_{n,t} = -\frac{1}{{n \choose t}-1}\quad \forall s \in [t].\]

In fact, we will ensure that \(|\lambda^{(w)}_{\alpha}| \leq C_t /n^{t+1}\) for all non-critical partitions \(\alpha\) of \(n\), where \(C_t >0\) depends only upon \(t\). This will imply that for \(n\) sufficiently large depending on \(t\), the \(\lambda_{\min}\)-eigenspace of \(A\) is precisely
\[\bigoplus_{s=1}^{t}U_{(n-s,s)}.\]
This will imply immediately (via Theorem \ref{thm:pseudo}) that if \(n\) is sufficiently large depending on \(t\), and \(\mathcal{A} \subset S_n\) is a \(t\)-set-intersecting family with \(|\mathcal{A}| = t!(n-t)!\), then the characteristic function of \(\mathcal{A}\) lies in the subspace \(U_t\) --- which, by Theorem \ref{thm:cosets-span}, is precisely the linear span of the characteristic functions of the \(T_{x \mapsto y}\)'s (\(x,y \in [n]^{(t)}\)). This will ultimately enable us to show that the \(T_{x \mapsto y}\)'s are the unique extremal families.

To summarize, our aim is to construct a class function satisfying the following conditions:

\begin{itemize}
\item \(w\) is supported on the \(t\)-derangements of \(S_n\);
\item \(\lambda^{(w)}_{(n)} = 1\);
\item \(\lambda^{(w)}_{(n-s,s)} = -\frac{1}{{n \choose t}-1}\ \forall s \in [t]\);
\item \(|\lambda^{(w)}_{\alpha}| \leq C_t/n^{t+1}\) for all other partitions \(\alpha\), where \(C_t >0\) depends only upon \(t\).
\end{itemize}

The following will enable us to deal with all the medium partitions at once.
\begin{claim}
\label{claim:mediumevals}
Let \(c_t\) be the constant in Lemma \ref{lemma:highdimension}. If \(\max_{\sigma \in S_n} |w(\sigma)| \leq K_t/n!\), then for all medium partitions \(\alpha\) of \(n\),
\[|\lambda_{\alpha}| \leq \frac{K_t}{c_t n^{t+1}}.\]
\end{claim}

To prove this, we will appeal to the following well-known fact:
\begin{lemma}
\label{lemma:trace}
If \(A\) is a real, symmetric, \(N \times N\) matrix, with eigenvalues \(\lambda_1,\lambda_2,\ldots,\lambda_N\) (repeated with their geometric multiplicities), then
\[\sum_{i=1}^{N} \lambda_i^2 = \Trace(A^2) = \sum_{i,j} A_{i,j}^2.\]
\end{lemma}
\begin{proof}[Proof of Lemma \ref{lemma:trace}]
Diagonalize \(A\): there exists a real invertible matrix \(P\) such that \(A = P^{-1}DP\), where \(D\) is the diagonal matrix
\[D = \left(\begin{array}{cccc} \lambda_{1} & 0 & \ldots & 0\\
0 & \lambda_{2} && 0\\
\vdots & & \ddots & \vdots\\
0 & & \ldots & \lambda_{N} \end{array} \right).\] 
We have \(A^{2} = P^{-1}D^{2}P\), and therefore
\[\sum_{i,j=1}^{N}A_{i,j}^2 = \Trace(A^{2}) = \Trace(P^{-1}D^{2}P) = \Trace(D^{2}) = \sum_{i=1}^{N}\lambda_{i}^{2},\]
as required.
\end{proof}

\begin{proof}[Proof of Claim \ref{claim:mediumevals}]
We apply Lemma \ref{lemma:trace} to the matrix \(A^{(w)}\). Since \(\dim(U_\alpha) = (f^\alpha)^2\), we obtain:
\[\sum_{\alpha \vdash n} (f^\alpha)^2 (\lambda^{(w)}_\alpha)^2 = \sum_{\sigma,\pi}A_{\sigma,\pi}^2 = n! \sum_{\tau \in S_n} w(\tau)^2 \leq (n!)^2 (K_t/n!)^2 = K_t^2.\]
It follows that for each partition \(\alpha\) of \(n\), we have
\[|\lambda^{(w)}_{\alpha}| \leq K_t/f^\alpha.\]
By Lemma \ref{lemma:highdimension}, there exists \(c_t>0\) such that for every medium partition \(\alpha\),
\[f^\alpha \geq c_t n^{t+1}.\]
Hence,
\[|\lambda^{(w)}_\alpha| \leq \frac{K_t}{c_t n^{t+1}}\]
for each medium partition \(\alpha\), proving the claim.
\end{proof}

Hence, we need only satisfy the following conditions:

\begin{enumerate}
\item \(w\) is supported on the \(t\)-derangements of \(S_n\);
\item \(\lambda^{(w)}_{(n)} = 1\);
\item \(\lambda^{(w)}_{(n-s,s)} = -\frac{1}{{n \choose t}-1}\ \forall s \in [t]\);
\item \(|\lambda^{(w)}_{\alpha}| \leq C_t/n^{t+1}\) for all fat, non-critical partitions \(\alpha\) (i.e. for all \(\alpha \in \mathcal{F}_{n,t} \setminus \mathcal{C}_{n,t}\)), where \(C_t\) depends only upon \(t\);
\item \(\max_{\sigma \in S_n} |w(\sigma)| \leq O_t(1/n!)\) (this deals with all the medium partitions);
\item \(\lambda^{(w)}_{\alpha} = 0\) for all tall partitions \(\alpha\).
\end{enumerate}

Rephrasing these in terms of the inner product, we obtain:

\begin{enumerate}
\item \(w\) is supported on the \(t\)-derangements of \(S_n\);
\item \(\langle w , \chi_{(n)} \rangle = 1\);
\item \(\langle w , \chi_{(n-s,s)} \rangle = -\frac{{n \choose s}-{n \choose s-1}}{{n \choose t}-1}\ \forall s \in [t]\);
\item \(|\langle w , \chi_{\alpha} \rangle| \leq f^{\alpha} C_t/n^{t+1}\) for all \(\alpha \in \mathcal{F}_{n,t} \setminus \mathcal{C}_{n,t}\), where \(C_t\) depends only upon \(t\);
\item \(\max_{\sigma \in S_n} |w(\sigma)| \leq O_t(1/n!)\);
\item \(\langle w , \chi_{\alpha} \rangle = 0\) for all tall partitions \(\alpha\).
\end{enumerate}

(In 3, we have used the fact that \(f^{(n-s,s)} = {n \choose s}-{n \choose s-1}\), from Corollary \ref{corr:dims}.) We now observe the following.
\begin{lemma}
\label{lemma:determine}
If \(w\) is any function supported on the \(t\)-derangements of \(S_n\), then \((\langle w,\chi_{\alpha}\rangle:\ \alpha_1 \geq n-t+1)\) determines \((\langle w, \chi_{\alpha} \rangle: \alpha_1=n-t)\).
\end{lemma}
\begin{proof}
Recall from Corollary \ref{corr:uppertriangular} that for any partition \(\gamma\) of \(n\), we have
\[\textrm{Span}\{\chi_{\alpha}:\ \alpha \geq \gamma\} = \textrm{Span}\{\xi_{\alpha}:\ \alpha \geq \gamma\}.\]
Observe that for any \(s \geq 0\),
\[\alpha_1 \geq n-s\quad \Leftrightarrow \quad \alpha \geq (n-s,1^{s}).\]
It follows that for {\em any} function \(w\), \((\langle w,\chi_{\alpha}\rangle:\ \alpha_1 \geq n-t+1)\) determines \((\langle w,\xi_{\alpha}\rangle:\ \alpha_1 \geq n-t+1)\). Recall that for any partition \(\alpha\) of \(n\), and any \(\sigma \in S_n\), \(\xi_{\alpha}(\sigma)\) is simply the number of \(\alpha\)-tabloids fixed by \(\sigma\). If \(w\) is supported on the \(t\)-derangments of \(S_n\), then \(\langle w, \xi_{\alpha} \rangle = 0\) whenever \(\alpha_1 = n-t\), since in this case a \(t\)-derangement cannot fix any \(\alpha\)-tabloid. Hence, \((\langle w,\chi_{\alpha}\rangle:\ \alpha_1 \geq n-t+1)\) determines \((\langle w,\xi_{\alpha}\rangle:\ \alpha_1 \geq n-t)\). But for any function \(w\), \((\langle w,\xi_{\alpha}\rangle:\ \alpha_1 \geq n-t)\) in turn determines \((\langle w,\chi_{\alpha}\rangle:\ \alpha_1 \geq n-t)\). It follows that if \(w\) is any function supported on the \(t\)-derangements of \(S_n\), then \((\langle w,\chi_{\alpha}\rangle:\ \alpha_1 \geq n-t+1)\) determines \((\langle w, \chi_{\alpha} \rangle: \alpha_1=n-t)\), as required.
\end{proof}

Note that if conditions 1 and 2 hold, and condition 3 holds for all \(s \in [t-1]\), then condition 3 holds for \(s=t\) also, since by Corollary \ref{corr:dims},
\[\sum_{s=0}^{t}\langle w, \chi_{(n-s,s)}\rangle = \langle w, \xi_{(n-t,t)}\rangle = 0.\]
It turns out that if, in addition, \(\langle w , \chi_{\alpha} \rangle = 0\) for all \(\alpha \in \mathcal{F}_{n,t-1} \setminus \mathcal{C}_{n,t}\), then condition 4 is satisfied by all \(\alpha \in \mathcal{F}_{n,t} \setminus \mathcal{C}_{n,t}\). This is the content of the following.

\begin{claim}
\label{claim:equiv}
If \(w\) is any class function satisfying
\begin{itemize}
 \item \(w\) is supported on the \(t\)-derangements of \(S_n\);
\item \(\langle w , \chi_{(n)} \rangle = 1\);
\item \(\langle w , \chi_{(n-s,s)} \rangle = -\frac{{n \choose s}-{n \choose s-1}}{{n \choose t}-1}\ \forall s \in [t-1]\);
\item \(\langle w , \chi_{\alpha} \rangle =0\) for all \(\alpha \in \mathcal{F}_{n,t-1} \setminus \mathcal{C}_{n,t}\),
\end{itemize}
then
\[\max\{|\langle w , \chi_{\alpha} \rangle|/f^\alpha:\ \alpha_1=n-t,\ \alpha \neq (n-t,t)\} \leq O_t(n^{-(t+1)}).\]
\end{claim}

\begin{proof}[Proof of Claim \ref{claim:equiv}:]
Let \(w\) be a class function satisfying the conditions of the claim. We first calculate the inner products \(\langle w,\xi_{\beta}\rangle:\ \beta \in \mathcal{F}_{n,t-1}\). We have
\[\langle w , \xi_{\beta} \rangle = 1-\sum_{s=0}^{t-1} K_{(n-s,s),\beta} \frac{{n \choose s}-{n \choose s-1}}{{n \choose t}-1} = 1-\epsilon_{n,t,\beta}\quad (\beta \in \mathcal{F}_{n,t-1}),\]
where we define
\[\epsilon_{n,t,\beta} = \sum_{s=0}^{t-1} K_{(n-s,s),\beta} \frac{{n \choose s}-{n \choose s-1}}{{n \choose t}-1}\quad (\beta \in \mathcal{F}_{n,t-1}).\]
By Proposition \ref{prop:independent}, for \(n \geq 2(t-1)\), the top-left minor of the Kostka matrix \(K\) indexed by the partitions in \(\mathcal{F}_{n,t-1}\) is independent of \(n\). It follows that for each \(t\), there exists \(D_t >0\) such that for all \(n\) and all \(\beta \in \mathcal{F}_{n,t-1}\),
\[K_{(n-s,s),\beta} \leq D_t.\]
Therefore, if \(t \in \mathbb{N}\) is fixed, we have
\[\max\{|\epsilon_{n,t,\beta}|:\ \beta \in \mathcal{F}_{n,t-1}\} = O_t(1/n).\]

As observed in the proof of the previous claim, we have \(\langle w,\xi_\beta \rangle = 0\) for all partitions \(\beta\) with \(\beta_1=n-t\). This will enable us to express the inner products \((\langle w, \chi_{\alpha} \rangle: \alpha_1=n-t)\) in terms of the inner products \((\langle w,\xi_{\beta}\rangle:\ \beta_1 \geq n-t+1)\).

Recall the determinantal formula (\ref{eq:determinantalformula}) for expressing an irreducible character \(\chi_{\alpha}\) in terms of the permutation characters \(\{\xi_{\beta}:\ \beta \geq \alpha\}\):
\[\chi_{\alpha} = \sum_{\pi \in S_{n}} \sgn(\pi) \xi_{\alpha - \textrm{id}+\pi}.\]

Let \(\alpha\) be any partition of \(n\) with \(\alpha_{1} = n-t\), \(\alpha \neq (n-t,t)\). Then we have
\[\langle w, \chi_{\alpha} \rangle = \sum_{\pi \in S_{n}} \sgn(\pi)\langle w , \xi_{\alpha - \textrm{id}+\pi} \rangle.\]
Note that \(\xi_{\alpha - \textrm{id}+\pi} = 0\) if \(\pi\) does not fix \(\{t+2,t+3,\ldots,n\}\) pointwise, i.e. if \(\pi \notin S_{[t+1]}\). Indeed, if \(\pi \notin S_{[t+1]}\), then \(\pi(j) < j\) for some \(j \geq t+2\). Since \(\alpha_j = 0\), we have \((\alpha + \textrm{id} +\pi)(j) = -(j-\pi(j)) <0\), so \(\xi_{\alpha+\textrm{id}+\pi} = 0\). Moreover, \(\langle w, \xi_{\alpha - \textrm{id}+\pi}\rangle = 0\) if \(\alpha_j - j + \pi(j) <0\) for some \(j \geq 2\), or if \(\pi(1)=1\) (since in this case, \((\alpha-\textrm{id}+\pi)(1) = \alpha_1 = n-t\)); otherwise, \(\langle w, \xi_{\alpha - \textrm{id}+\pi}\rangle = 1 - \epsilon_{n,t,\beta}\) for some \(\beta \in \mathcal{F}_{n,t-1}\). Hence,
\begin{align*}
\langle w, \chi_{\alpha} \rangle & = \sum_{\pi \in S_{[t+1]}} \sgn(\pi)\langle w, \xi_{\alpha - \textrm{id}+ \pi} \rangle\\
& = \sum_{\pi \in S_{[t+1]}} \sgn(\pi)(f_{\alpha}(\pi) - O_{t}(1/n))\\
& = \sum_{\pi \in S_{[t+1]}} \sgn(\pi)f_{\alpha}(\pi) - O_t(1/n),
\end{align*} 
where we define
\[f_{\alpha}(\pi) = 1\{\pi(1) \neq 1\textrm{ and }\alpha_j-j+\pi(j) \geq 0 \textrm{ for all }j \geq 2\}.\]
We will now show that
\begin{equation}\label{eq:determinantzero} \sum_{\pi \in S_{[t+1]}} \sgn(\pi)f_{\alpha}(\pi) = 0.\end{equation}
Since \(\alpha_1=n-t\), and \(\alpha \neq (n-t,t)\), we must have \(\alpha_3 \geq 1\). We split into 2 cases.\\
\\
\textit{Case (i):} \(\alpha_3\geq 2\).\\
In this case, (\ref{eq:determinantzero}) follows from
\begin{equation}\label{eq:23swap} f_{\alpha}(\pi) = f_{\alpha}(\pi(2\ 3)) \quad \forall \pi \in S_{n}.\end{equation}
To prove (\ref{eq:23swap}), note firstly that \(\pi(1)=1\) if and only if \(\pi(2\ 3)(1)=1\). Secondly, since \(\alpha_2 \geq \alpha_3 \geq 2\), we have \(\alpha_2-2+\rho(2) \geq \rho(2) \geq 1\) and \(\alpha_3 - 3+\rho(3) \geq \rho(3)-1 \geq 0\) for all permutations \(\rho \in S_{n}\). Hence,
\[\alpha_j-j+\pi(j) \geq 0 \ \forall j \geq 2 \quad \Leftrightarrow \quad \alpha_j-j+\pi(2\ 3)(j) \geq 0 \ \forall j \geq 2,\]
proving (\ref{eq:23swap}).\\
\\
\textit{Case (ii):} \(\alpha_{3} = 1\).\\
In this case, (\ref{eq:determinantzero}) follows from
\begin{equation}\label{eq:13swap} f_{\alpha}(\pi) = f_{\alpha}(\pi(1\ 3))\quad \forall \pi \in S_n.\end{equation}
To prove (\ref{eq:13swap}), let \(\pi \in S_n\) with \(f_{\alpha}(\pi) = 0\); then \(\pi(1)=1\) or \(\alpha_j - j+\pi(j) <0\) for some \(j \geq 2\). If \(\pi(1)=1\), then \(\alpha_3 - 3 +\pi(1\ 3)(3) = -1 <0\), so \(f_{\alpha}(\pi(1\ 3)) = 0\). If \(\alpha_j - j+\pi(j) <0\) for some \(j \geq 2\), \(j \neq 3\), then \(\alpha_j - j+\pi(1\ 3)(j) <0\), so \(f_{\alpha}(\pi(1\ 3)) = 0\). If \(\alpha_3-3+\pi(3) <0\), then we must have \(\pi(3)=1\), so \(\pi(1\ 3)(1)=1\), and therefore \(f_{\alpha}(\pi(1\ 3)) = 0\). Hence,
\[f_{\alpha}(\pi) = 0\quad \Rightarrow \quad f_{\alpha}(\pi(1\ 3)) = 0\quad \forall \pi \in S_n.\]
Applying \((1\ 3)\) to both arguments gives
\[f_{\alpha}(\pi(1\ 3)) = 0\quad \Rightarrow \quad f_{\alpha}(\pi) = 0\quad \forall \pi \in S_n,\]
proving (\ref{eq:13swap}).

It follows that
\[\max\{|\langle w, \chi_{\alpha} \rangle|:\ \alpha_1=n-t,\ \alpha \neq (n-t,t)\} = O_t(1/n).\]
By Lemma \ref{lemma:dimension-for-long}, we have
\[f^{\alpha} > e^{-t} {n \choose t}\]
for any partition \(\alpha\) of \(n\) such that \(\alpha_{1}=n-t\). Hence,
\[\max\{|\langle w, \chi_{\alpha} \rangle|/f^{\alpha}:\ \alpha_1=n-t,\ \alpha \neq (n-t,t)\} \leq O_t(1/n),\]
proving the claim.
\end{proof}

Hence, it suffices for \(w\) to satisfy the following conditions:

\begin{enumerate}
\item \(w\) is supported on the \(t\)-derangements of \(S_n\);
\item \(\langle w , \chi_{(n)} \rangle = 1\);
\item \(\langle w , \chi_{(n-s,s)} \rangle = -\frac{{n \choose s}-{n \choose s-1}}{{n \choose t}-1}\ \forall s \in [t-1]\);
\item \(\langle w , \chi_{\alpha} \rangle = 0\) for all \(\alpha \in \mathcal{F}_{n,t-1} \setminus \mathcal{C}_{n,t}\);
\item \(\max_{\sigma \in S_n} |w(\sigma)| \leq O_t(1/n!)\);
\item \(\langle w , \chi_{\alpha} \rangle = 0\) for all tall partitions \(\alpha\).
\end{enumerate}

In order to satisfy these conditions, we will in fact show that there exist two class functions \(w^{+}\) and \(w^{-}\), supported on even and odd permutations respectively, each satisfying conditions 1--5; \(w = \tfrac{1}{2}(w^{+}+w^{-})\) will then satisfy all 6 conditions. By Lemma \ref{lemma:determine}, conditions 1--4 determine \((\langle w, \chi_{\alpha} \rangle:\ \alpha_1 \geq n-t)\), so if \(w^{+}\) and \(w^{-}\) both satisfy conditions 1--4, then
\[\langle w^{+}, \chi_{\alpha} \rangle = \langle w^{-}, \chi_{\alpha}\rangle \quad \textrm{for each }\alpha \in \mathcal{F}_{n,t}.\]
By Theorem \ref{thm:transpose-sign}, for any partition \(\alpha\), we have \(\chi_{\alpha^{\top}} = \chi_{\alpha} \cdot \sgn\). Recall that the tall partitions are precisely the transposes of the fat partitions. For any fat partition \(\alpha\), we have
\begin{align*}
\langle w, \chi_{\alpha^{\top}} \rangle & = \langle w, \chi_{\alpha} \cdot \sgn \rangle\\
& = \tfrac{1}{2}(\langle w^{+}, \chi_{\alpha}\cdot \sgn \rangle + \langle w^{-}, \chi_{\alpha} \cdot \sgn\rangle)\\
&= \tfrac{1}{2}(\langle w^{+}, \chi_{\alpha} \rangle - \langle w^{-}, \chi_{\alpha} \rangle)\\
&=0, 
\end{align*}
so condition 6 is satisfied.

It remains to show that there exist two class functions \(w^{+}\) and \(w^{-}\), supported on respectively even and odd permutations, each satisfying conditions 1--5. Equivalently, in terms of the permutation characters \(\{\xi_{\beta}:\ \beta \in \mathcal{F}_{n,t-1}\}\), we seek to satisfy the following:

\begin{itemize}
\item \(w^{\pm}\) is supported on the \(t\)-derangements of \(S_n\);
\item \(\langle w^{\pm} , \xi_{\beta} \rangle = 1-\sum_{s=0}^{t-1} K_{(n-s,s),\beta} \frac{{n \choose s}-{n \choose s-1}}{{n \choose t}-1}\) for all \(\beta \in \mathcal{F}_{n,t-1}\);
\item \(\max_{\sigma \in S_n} |w^{\pm}(\sigma)| \leq O_t(1/n!)\).
\end{itemize}
\begin{flushright}\((*)\)\end{flushright}
As before, for each partition \(\beta \in \mathcal{F}_{n,t-1}\), let
\[\epsilon_{n,t,\beta} = \sum_{s=0}^{t-1} K_{(n-s,s),\beta} \frac{{n \choose s}-{n \choose s-1}}{{n \choose t}-1};\]
let
\[\eta_{n,t,\beta} = 1-\epsilon_{n,t,\beta}.\]
The second condition above becomes
\begin{equation}
 \label{eq:cond}
\langle w^{\pm} , \xi_{\beta} \rangle = \eta_{n,t,\beta} \quad \forall \beta \in \mathcal{F}_{n,t-1}.
\end{equation}

Recall that
\[\max\{|\epsilon_{n,t,\beta}|:\ \beta \in \mathcal{F}_{n,t-1}\} = O_t(1/n),\]
so if \(n\) is sufficiently large depending upon \(t\), then \(\eta_{n,t,\beta} \in [0,1]\) for all \(\beta \in \mathcal{F}_{n,t-1}\).

Let \(\tilde{N} = (\xi_{\beta}(X_\alpha))_{\beta,\alpha \in \mathcal{F}_{n,t-1}}\) denote the top-left minor of the permutation character table of \(S_{n}\), indexed by the partitions in \(\mathcal{F}_{n,t-1}\). By Corollary \ref{corr:uppertriangular}, \(\tilde{N}\) is independent of \(n\) for \(n > 2(t-1)\). Suppose \(u\) is a class function supported only on the conjugacy classes with cycle-type in \(\mathcal{F}_{n,t-1}\). For each \(\alpha \in \mathcal{F}_{n,t-1}\), let \(u_{\alpha}\) be the value taken by \(u\) on \(X_\alpha\). Then
\[\langle u, \xi_{\beta} \rangle = \sum_{\alpha \in \mathcal{F}_{n,t-1}} |X_{\alpha}| u_{\alpha}\tilde{N}_{\beta, \alpha} \quad (\beta \in \mathcal{F}_{n,t-1}).\]
Since \(\tilde{N}\) is invertible, there is a unique choice of the \(u_{\alpha}\)'s satisfying (\ref{eq:cond}). However, there are two problems: we cannot choose \(u\) to be supported only on even permutations (or only on odd permutations), and since \(|X_\alpha| \leq n!/(n-t+1)\) for each \(\alpha \in \mathcal{F}_{n,t-1}\), we cannot demand that \(\max\{|u_\alpha|:\ \alpha \in \mathcal{F}_{n,t-1}\} \leq O_t(1/n!)\). Indeed, if \(|u_\alpha| \leq K_t/n!\) for each \(\alpha \in \mathcal{F}_{n,t-1}\), then 
\[\left|\sum_{\alpha \in \mathcal{F}_{n,t-1}} |X_{\alpha}| u_{\alpha}\tilde{N}_{\beta, \alpha} \right| \leq \frac{K_t |\mathcal{F}_{n,t-1}|}{n-t+1} < \frac{2K_t \sum_{s=0}^{t-1} p_s}{n} = O_t(1/n)\]
for all \(\beta \in \mathcal{F}_{n,t-1}\), provided \(n > 2(t-1)\); this is \(< 1-\epsilon_{n,t,\beta}\) for all \(\beta \in \mathcal{F}_{n,t-1}\), provided \(n\) is sufficiently large.

We solve both problems simultaneously by exhibiting, for each \(\alpha \in \mathcal{F}_{n,t-1}\), a large collection of conjugacy classes of \(S_n\) which fix no \(t\)-set, and on which each permutation character \(\xi_{\beta}:\ \beta \in \mathcal{F}_{n,t-1}\) takes the same value as it does on \(X_{\alpha}\). These conjugacy classes can be chosen to be all even (or all odd), with total size at least \(B_t n!\) for some \(B_t>0\) depending upon \(t\) alone.

We now define these collections. Let \(\alpha \in \mathcal{F}_{n,t-1}\); suppose \(\alpha = (n-s,\alpha_{2},\ldots,\alpha_{l})\), where \(0 \leq s \leq t-1\) and \(\alpha_{l} \geq 1\). Let \(\mathcal{P}_{n,t}(\alpha)\) be the collection of partitions that can be obtained from \(\alpha\) by subdividing the part of size \(n-s\) into parts of size \(> t\). Let \(\mathcal{S}_{n,t}(\alpha)\) be the family of permutations with cycle-type in \(\mathcal{P}_{n,t}(\alpha)\); let \(\mathcal{S}^{+}_{n,t}(\alpha)\) and \(\mathcal{S}_{n,t}^{-}(\alpha)\) be the families of respectively even and odd permutations in \(\mathcal{S}_{n,t}(\alpha)\). Notice that \(\{\mathcal{P}_{n,t}(\alpha): \alpha \in \mathcal{F}_{n,t-1}\}\) are pairwise disjoint, and that \(\mathcal{S}_{n,t}^{+}(\alpha)\) and \(\mathcal{S}_{n,t}^{-}(\alpha)\) are unions of conjugacy classes of \(t\)-derangements.

For each \(j \in \mathbb{N}\), let \(a_j\) be the number of parts of \(\alpha\) of size \(j\). Observe that
\begin{align*}
\min(|\mathcal{S}_{n,t}^{+}(\alpha)|,|\mathcal{S}_{n,t}^{-}(\alpha)|) & = \frac{1}{\alpha_{2}\ldots\alpha_{l}\prod_{j} a_{j}!} \frac{n!}{(n-s)!} \min\{E_{n-s,t},O_{n-s,t}\}\\
& \geq \frac{L_{t}n!}{\alpha_{2}\ldots\alpha_{l}\prod_{j} a_{j}!}.
\end{align*}
Here, the equality follows from the fact that to choose a permutation in \(\mathcal{S}_{n,t}^{+}(\alpha)\) or \(\mathcal{S}_{n,t}^{-}(\alpha)\), we can first choose the \(\alpha_{2},\ldots,\alpha_{l}\)-cycles, and then choose an appropriately signed permutation of the other \(n-s\) numbers which has no cycle of length \(\leq t\). The inequality follows from Lemma \ref{lemma:noshortcyclesbound}. By the AM/GM inequality, we have
\[\alpha_{2}\ldots\alpha_{l}\prod_{j} a_{j}! \leq (s/l)^l l! \leq s^l \leq (t-1)^{t-1},\]
and therefore
\[\min(|\mathcal{S}_{n,t}^{+}(\alpha)|,|\mathcal{S}_{n,t}^{-}(\alpha)|) \geq L_t n!/(t-1)^{t-1} = B_t n!,\]
where \(B_t >0\) depends upon \(t\) alone.

We must now show that for each \(\beta \in \mathcal{F}_{n,t-1}\), the permutation character \(\chi_{\beta}\) takes the value \(\xi_{\beta}(X_\alpha)\) on all of \(\mathcal{S}_{n,t}(\alpha)\). Take any permutation \(\sigma \in \mathcal{S}_{n,t}(\alpha)\), where \(\alpha = (n-s,\alpha_{2},\ldots,\alpha_{l})\) with \(0 \leq s \leq t-1\), and let \(\sigma'\) be the permutation with cycle-type \(\alpha\) produced from \(\sigma\) by merging all cycles of length \(> t\). Take any partition \(\beta = (n-r,\beta_{2},\ldots,\beta_{k})\) with \(0 \leq r \leq t-1\). We wish to evaluate \(\xi_{\beta}(\sigma)\), the number of \(\beta\)-tabloids fixed by \(\sigma\). Note that \(\sigma\) and \(\sigma'\) fix the same \(\beta\)-tabloids, since if \([T]\) is a \(\beta\)-tabloid fixed by \(\sigma\), then all numbers in \((>t)\)-cycles of \(\sigma\) must occur in the first row of \([T]\), so \([T]\) must be fixed by \(\sigma'\) as well. Hence,
\[\xi_{\beta}(\sigma) = \xi_{\beta}(\sigma') = \xi_{\beta}(X_{\alpha}),\]
as desired.

We can now define the class functions \(w^{+}\) and \(w^{-}\). We will let \(w^{+}\) have value \(x_{\gamma}^{+}\) on \(\mathcal{S}_{n,t}^{+}(\gamma)\) for each partition \(\gamma \in \mathcal{F}_{n,t-1}\), and zero elsewhere; to satisfy condition (\ref{eq:cond}), we need 
\[\sum_{\gamma \in \mathcal{F}_{n,t-1}} |\mathcal{S}_{n,t}^{+}(\gamma)|x^{+}_{\gamma}  \tilde{N}_{\beta,\gamma} = \eta_{n,t,\beta}\quad \forall \beta \in \mathcal{F}_{n,t-1},\]
i.e.
\[x^{+}_{\gamma} = \sum_{\beta \in \mathcal{F}_{n,t-1}} (\tilde{N}^{-1})_{\gamma, \beta} \eta_{n,t,\beta}/ |\mathcal{S}^{+}_{n,t}(\gamma)|\quad \forall \gamma \in \mathcal{F}_{n,t-1}.\]
Similarly, we will let \(w^{-}\) have value \(x_{\gamma}^{-}\) on \(\mathcal{S}_{n,t}^{-}(\gamma)\) for each partition \(\gamma \in \mathcal{F}_{n,t-1}\), and zero elsewhere; to satisfy (\ref{eq:cond}), we need 
\[\sum_{\gamma \in \mathcal{F}_{n,t-1}} |\mathcal{S}_{n,t}^{-}(\gamma)| x^{-}_{\gamma} \tilde{N}_{\beta,\gamma} = \eta_{n,t,\beta}\quad \forall \beta \in \mathcal{F}_{n,t-1},\]
i.e.
\[x^{-}_{\gamma} = \sum_{\beta \in \mathcal{F}_{n,t-1}} (\tilde{N}^{-1})_{\gamma, \beta} \eta_{n,t,\beta}/ |\mathcal{S}^{-}_{n,t}(\gamma)|\quad \forall \gamma \in \mathcal{F}_{n,t-1}.\]
By construction, \(w^{+}\) and \(w^{-}\) are class functions supported on respectively even and odd \(t\)-derangments, and satisfy (\ref{eq:cond}). It remains to show that \(|w^{\pm}(\sigma)| \leq K_t/n!\) for all \(\sigma \in S_n\), for some \(K_t\) depending only upon \(t\). Since \(\tilde{N}\) is invertible and independent of \(n\) for \(n > 2(t-1)\), \(\tilde{N}^{-1}\) is also independent of \(n\) for \(n > 2(t-1)\), so we have
\[\max\{|(\tilde{N}^{-1})_{\gamma,\beta}|:\ \beta,\gamma \in \mathcal{F}_{n,t-1}\} \leq O_t(1).\]
We have \(|\mathcal{F}_{n,t-1}| \leq O_t(1)\), and
\[|\mathcal{S}^{+}_{n,t}(\gamma)|,|\mathcal{S}^{-}_{n,t}(\gamma)| \geq B_t n!\quad \forall \gamma \in \mathcal{F}_{n,t-1},\]
where \(B_t>0\) depends only upon \(t\). Finally, we have \(\eta_{n,t,\beta} \in [0,1]\) for all \(\beta \in \mathcal{F}_{n,t-1}\), provided \(n\) is sufficiently large depending upon \(t\). It follows that
\[\max\{|x^{\pm}_\gamma|:\ \gamma \in \mathcal{F}_{n,t-1}\} \leq O_t(1/n!),\]
as required. Hence, \(w = \tfrac{1}{2}(w^{+}+w^{-})\) satisfies all the conditions \((*)\).

Applying Theorem \ref{thm:pseudo} to the graph \(\Gamma_{(t)}\) and the pseudoadjacency matrix \(A^{(w)}\) constructed above, we obtain the upper bound in Theorem \ref{thm:main}, together with an algebraic characterization of the extremal families.
\begin{theorem}
\label{thm:upperbound}
If \(n\) is sufficiently large depending on \(t\), and \(\mathcal{A} \subset S_n\) is \(t\)-set-intersecting, then
\[|\mathcal{A}| \leq t!(n-t)!.\]
If equality holds, then the characteristic function of \(\mathcal{A}\) lies in the critical subspace \(U_t\).
\end{theorem}
Recall from Theorem \ref{thm:cosets-span} that \(U_t\) is precisely the subspace of \(\mathbb{C}[S_n]\) spanned by the characteristic functions of cosets of stabilizers of \(t\)-sets.

\section{The extremal families}

Our aim in this section is to prove the following.
\begin{theorem}
\label{thm:equality}
If \(n\) is sufficiently large depending on \(t\), and \(\mathcal{A} \subset S_n\) is \(t\)-set-intersecting with \(|\mathcal{A}| = t!(n-t)!\), then \(\mathcal{A}\) is a coset of the stabilizer of a \(t\)-set.
\end{theorem}

By Theorem \ref{thm:upperbound}, if \(n\) is sufficiently large depending on \(t\), and \(\mathcal{A} \subset S_n\) is \(t\)-set-intersecting with \(|\mathcal{A}| = t!(n-t)!\), then the characteristic function of \(\mathcal{A}\) is in \(U_t\), the linear span of the characteristic functions of the \(T_{x \mapsto y}\)'s. Let \(V_t\) denote linear span of the characteristic functions of the {\em \(t\)-cosets}. Clearly, \(U_t \leq V_t\). Recall the following theorem from \cite{jointpaper}.
\begin{theorem}[Benabbas, Friedgut, Pilpel]
\label{thm:bfp}
If the characteristic function of \(\mathcal{A} \subset S_n\) is a linear combination of the characteristic functions of the \(t\)-cosets of \(S_n\), then \(\mathcal{A}\) is a disjoint union of \(t\)-cosets of \(S_n\).
\end{theorem}
It follows that a family \(\mathcal{A} \subset S_n\) satisfying the hypotheses of Theorem \ref{thm:equality} must be a disjoint union of \(t\)-cosets of \(S_n\). Hence, it suffices to prove the following
\begin{proposition}
\label{prop:equality}
For any \(n \geq 2t\), if \(\mathcal{A} \subset S_n\) is \(t\)-set-intersecting with \(|\mathcal{A}|=t!(n-t)!\), and is a disjoint union of \(t\)-cosets, then it is a coset of the stabilizer of a \(t\)-set.
\end{proposition}
\begin{proof}
Note that for any \(\pi,\tau \in S_n\), we may replace \(\mathcal{A}\) by the `double translate' \(\pi \mathcal{A} \tau\); this is a disjoint union of \(t\)-cosets if and only if \(\mathcal{A}\) is, and is a coset of the stabilizer of a \(t\)-set if and only if \(\mathcal{A}\) is.

Suppose \(C=T_{a_1 \mapsto b_1,\ldots, a_t \mapsto b_t} \subset \mathcal{A}\). Without loss of generality, by double translation, we may assume that \(a_l = b_l = l\) for each \(l \in [t]\), i.e. \(C=T_{1 \mapsto 1,\ldots,t \mapsto t} \subset \mathcal{A}\). Our aim is to show that \(\mathcal{A} = \{\sigma \in S_n:\ \sigma([t])=[t]\}\).

Suppose for a contradiction that \(\mathcal{A}\) contains a \(t\)-coset \(C' = T_{i_1 \mapsto p(i_1),\ldots,i_t \mapsto p(i_t)}\) which does not map \([t]\) to itself. We will show that there exist permutations \(\sigma \in C\) and \(\pi \in C'\) such that no \(t\)-set has the same image under \(\sigma\) and \(\pi\).

Let \(I = \{i_1,i_2,\ldots,i_t\}\); we regard \(p:I \to [n]\) as an injection. Let \(J = p(I) = \{p(i_1),p(i_2),\ldots,p(i_t)\}\), let \(I' = I \cap [t]\), let \(I'' = I \setminus [t]\), let \(R = [t] \setminus I\), and let \(B = p(I'')\). We will need the following.

\begin{claim}
\label{claim:product}
Let \(\tau \in S_n\) be any permutation satisfying
\begin{equation}
\label{eq:tauconditions}
\tau(i) = p(i)\ \forall i \in I',\quad \tau^{-1}(B) \subset [t]^{c}.
\end{equation}
Then there exist \(\sigma \in C\) and \(\pi \in C'\) such that \(\pi \sigma^{-1} = \tau\).
\end{claim}
\begin{proof}[Proof of Claim \ref{claim:product}:]
Let \(\tau \in S_n\) be as above. Choose any \(\sigma \in C\) such that \(\sigma(i) = \tau^{-1}(p(i))\) for all \(i \in I''\). (We can do this because \(I''\) and \(\tau^{-1}(B)\) are both subsets of \([t]^{c}\).) Then \(\tau \sigma(i) = p(i)\) for all \(i \in I''\), and \(\tau \sigma(i) = \tau(i) = p(i)\) for all \(i \in I'\). Hence, \(\tau \sigma(i) = p(i)\) for all \(i \in I\), so \(\tau \sigma \in C'\). This proves the claim.
\end{proof}

Since we are assuming that \(C'\) does not fix \([t]\), we must have either \(I \neq [t]\) or \(J \neq [t]\). We split into two cases:\\
\\
\textit{Case (i):} \(I \cap [t] = J \cap [t] = \emptyset\).\\
If \(n=2t\), then \(I=J=[t]^{c}\), so \(C'\) maps \([t]\) to itself, contrary to our assumption, so we may assume that \(n > 2t\). Without loss of generality, by double translation, we may assume that \(C' = T_{t+1 \mapsto t+1,\ldots,2t \mapsto 2t}\). Observe that
\begin{align*}
\sigma & := (2t+1,\ 2t,\ 2t-1,\ldots, t+2,\ t+1) \in C,\\
\pi & := (1,\ 2,\ldots t,\ 2t+1,\ 2t+2,\ldots, n) \in C',
\end{align*}
but
\[\sigma^{-1} \pi = (1\ 2\ 3\ldots n)\]
is an \(n\)-cycle, and therefore fixes no \(t\)-set. Hence, no \(t\)-set has the same image under \(\sigma\) and \(\pi\), contradicting the fact that \(\mathcal{A}\) is \(t\)-set-intersecting.\\
\\
\textit{Case (ii):} \(I \cap [t] \neq \emptyset\) or \(J \cap [t] \neq \emptyset\).\\
We need the following.
\begin{claim}
\label{claim:mapsoutside}
If there exists \(i^* \in I \cap [t]\) such that \(p(i^*) \notin [t]\), then there exists a permutation \(\tau \in S_n\) satisfying (\ref{eq:tauconditions}), and with all the numbers in \([t]^{c} \cup \{i^*\}\) appearing in a single cycle.
\end{claim}
\begin{proof}[Proof of Claim \ref{claim:mapsoutside}:]
For each \(i \in I'\), consider the sequence of iterates
\[i,\ p(i),\ p(p(i)),\ \ldots\]
If the sequence remains within \(I'\), then it forms a cycle, which we call an `iteration cycle'; otherwise, it is part of a maximal chain of iterates which starts in \(I'\) and has all its members in \(I'\) except for the last one; we call this an `iteration chain'. The iteration cycles and chains are pairwise disjoint, and cover \(I'\).

We now construct a suitable permutation \(\tau \in S_n\). Each iteration cycle is a cycle of \(\tau\); the other numbers in \([n]\) are in a single long cycle of \(\tau\), formed from a long chain which we construct as follows.

Let \(S\) be the set of all start-points of iteration chains that end in \([t]^c\). By assumption, there exists \(i^* \in I'\) such that \(p(i^*) \notin [t]\), so at least one iteration chain ends in \([t]^{c}\), and therefore \(S \neq \emptyset\). Let \(\mathcal{F}\) be the set of iteration chains that start in \(B \cap S\) and end in \([t]^{c}\); note that we may have \(\mathcal{F} = \emptyset\).

We first form a special collection \(\mathcal{G}\) of disjoint chains, as follows.
\begin{itemize}
 \item String together the chains in \(\mathcal{F}\), end-to-end, producing a single chain starting in \(B \cap S\) and ending in \([t]^{c}\). Then add all the points of \(B \cap [t]^{c}\) onto the end of this chain, one by one, in any order, forming a `bad' chain starting in \(B \cap S\) and ending in \([t]^{c}\). Include the bad chain in \(\mathcal{G}\) (it may be extended later).
\item For each \(b \in B \cap R\) in turn, choose \(q_b \in [t]^{c}\) not already in a \(\mathcal{G}\)-chain. If \(q_b\) is the endpoint of an iteration chain, include the whole of this iteration chain in \(\mathcal{G}\); otherwise, include the two-point chain \((q_b,b)\) in \(\mathcal{G}\).
\item For each iteration chain \(\mathcal{C}\) starting in \(B \setminus S\) and ending in \([t] \setminus I = R\), choose a different point \(q_{\mathcal{C}} \in [t]^{c}\) which is not already a member of any \(\mathcal{G}\)-chain. If \(q_{\mathcal{C}}\) is the endpoint of an iteration chain \(\mathcal{D}\), include the chain \((\mathcal{D},\mathcal{C})\) in \(\mathcal{G}\); otherwise, include the chain \((q_{\mathcal{C}},\mathcal{C})\) in \(\mathcal{G}\).
\item Observe that we have used exactly \(|B|\) points of \([t]^{c}\) so far, all the elements of \(B\) are in \(\mathcal{G}\)-chains, and every element of \(B\) is preceded in its \(\mathcal{G}\)-chain by an element of \([t]^c\), except for the starting-point of the bad chain. (If the bad chain is empty, then {\em every} element of \(B\) is preceded in its chain by an element of \([t]^c\), and we go to the final step.) Now suppose the bad chain is non-empty. By assumption, \(i^* \in I'\), so \(|I''| < t\), and therefore \(|B| = |I''| < t \leq n-t\), so there is still at least one point of \([t]^{c}\) that has not been used in any \(\mathcal{G}\)-chain. Choose one such point, \(u\) say, and add it to the start of the bad chain. If \(u\) is the endpoint of an iteration chain, then add the whole of this iteration chain to the start of the bad chain. (Note that this iteration chain cannot start in \(B \cap S\), since the endpoints of all such iteration chains have already been used in \(\mathcal{G}\)-chains.) {\em Every} element of \(B\) is now preceded in its \(\mathcal{G}\)-chain by an element of \([t]^c\).
\item Finally, include in \(\mathcal{G}\) all the iteration chains with none of their elements previously used in \(\mathcal{G}\)-chains.
\end{itemize}
Observe that \(\mathcal{G}\) consists of disjoint chains that cover all the iteration chains, and all points of \(B\). We now string together all the \(\mathcal{G}\)-chains and all the remaining isolated points in \([t]^{c}\) to produce a long chain; this forms the long cycle of \(\tau\). This long cycle contains all the numbers in \([t]^{c} \cup \{i^*\}\), and by construction, \(\tau\) satisfies (\ref{eq:tauconditions}), proving the claim.
\end{proof}

Hence, if there exists \(i^* \in I \cap [t]\) such that \(p(i^*) \notin [t]\), then by Claim \ref{claim:product}, there exist \(\sigma \in C\), \(\pi \in C'\) such that \(\pi \sigma^{-1}\) has a cycle of length at least \(n-t+1\), so cannot fix any \(t\)-set, a contradiction.

Similarly, if there exists \(i^* \in I \setminus [t]\) such that \(p(i^*) \in [t]\), then we obtain a contradiction by taking inverses and applying the above argument.

Hence, from now on, we may assume that \(i \in [t] \Leftrightarrow p(i) \in [t]\) for each \(i \in I\). Therefore, \(|I \cap [t]| = |J \cap [t]|\). By assumption, we cannot have \(I \cap [t] = J \cap [t] = \emptyset\), or \(I = J = [t]\), and therefore \(|I \cap [t]| = |J \cap [t]| = s \neq 0,t\).

We now prove the following.
\begin{claim}
\label{claim:ifandonlyif}
If \(i \in [t] \Leftrightarrow p(i) \in [t]\) for each \(i \in I\), then there exists a permutation \(\tau \in S_n\) satisfying (\ref{eq:tauconditions}), and with all the numbers in \((I \cap [t])^{c}\) appearing in a single cycle.
\end{claim}
\begin{proof}[Proof of Claim \ref{claim:ifandonlyif}:]
Define the iteration chains and the iteration cycles as before. This time, since \(i \in [t] \Leftrightarrow p(i) \in [t]\) for each \(i \in I\), all the iteration chains end in \([t] \setminus I = R\). Let each iteration cycle be a cycle of \(\tau\); the other numbers in \([n]\) will be in a single long cycle of \(\tau\), formed from a long chain as follows.

Start at the beginning of any iteration chain. Join the iteration chains together end-to-end, in any order, to produce a single chain, and then add the other points of \([t] \setminus I = R\) onto the end of the chain, in any order. Note that \(B \subsetneq [t]^c\). Choose any point of \([t]^{c} \setminus B\), add it onto the end of the chain, and then add the rest of \([t]^c\) onto the end of the chain, producing a long chain. Note that this long chain covers all of \((I \cap [t])^c\), and each element of \(B\) is preceded in the long chain by an element of \([t]^c\). The long chain forms the long cycle of \(\tau\).

The iteration cycles all lie within \(I \cap [t]\); all the other numbers in \([n]\) are in a single cycle of \(\tau\). By construction, \(\tau\) satisfies (\ref{eq:tauconditions}), proving the claim.
\end{proof}

Hence, if \(i \in [t] \Leftrightarrow p(i) \in [t]\) for each \(i \in I\), then by Claim \ref{claim:product}, there exist \(\sigma \in C\), \(\pi \in C'\) such that \(\pi \sigma^{-1}\) has a cycle of length at least \(n-t+1\), so cannot fix any \(t\)-set, a contradiction.

It follows that any \(t\)-coset within \(\mathcal{A}\) must map \([t]\) to itself, so all the permutations in \(\mathcal{A}\) must map \([t]\) to itself. Since \(|\mathcal{A}|=t!(n-t)!\), it follows that \(\mathcal{A} = \{\sigma \in S_n:\ \sigma([t])=[t]\}\), proving Proposition \ref{prop:equality}.
\end{proof}

\section{Conclusion and open problems}
It is natural to ask for which values of \(n\) and \(t\) Theorem \ref{thm:main} holds. We conjecture that in fact, the upper bound in Theorem \ref{thm:main} holds for all \(n\) and \(t\), and that the unique extremal families are the cosets of stabilizers of \(t\)-sets, unless \(t=2\) and \(n=4\). (Observe that in the case \(t=2\) and \(n=4\), the family \[\{\textrm{id},(12)(34),(13)(24),(14)(23)\}\]
is also extremal, but is not a coset of the stabilizer of a 2-set.)
\begin{conjecture}
\label{conj:main}
For any positive integers \(t \leq n\), if \(\mathcal{A} \subset S_n\) is \(t\)-set-intersecting, then \(|\mathcal{A}| \leq t!(n-t)!\). Equality holds only if \(\mathcal{A}\) is a coset of the stabilizer of a \(t\)-set, unless \(t=2\) and \(n\in\{4,5\}\).
\end{conjecture}
We believe that new techniques will be required to prove this conjecture.

The situation is different for \(t\)-intersecting families: if \(t \geq 4\) and \(n \leq 2t\), then the \(t\)-intersecting family
\[\{\sigma \in S_n:\ \sigma \textrm{ has at least } t+1 \textrm{ fixed points in } [t+2]\}\]
is larger than a \(t\)-coset. In \cite{jointpaper}, the following conjecture is made.
\begin{conjecture}
A maximum-sized \(t\)-intersecting family in \(S_{n}\) must be a double translate of one of the families
\[M_{i} = \{\sigma \in S_{n}:\ \sigma \textrm{ has at least } t+i \textrm{ fixed points in } [t+2i]\}\ (0 \leq i \leq (n-t)/2),\]
i.e. of the form \(\pi M_i \tau\) where \(\pi,\tau \in S_n\).
\end{conjecture}
This would imply the following.
\begin{conjecture}
For any \(t \in \mathbb{N}\), if \(n > 2t\), the maximum-sized \(t\)-intersecting families in \(S_n\) are the \(t\)-cosets.
\end{conjecture}
Even this conjecture remains open.

\subsection*{Some related open problems}
Of course, a family \(\mathcal{A} \subset S_n\) is 2-set-intersecting if and only if for any two permutations \(\sigma,\pi \in \mathcal{A}\), there exists \(\{i,j\} \in [n]^{(2)}\) such that either \(\sigma(i)=\pi(i)\) and \(\sigma(j)=\pi(j)\), or \(\sigma(i)=\pi(j)\) and \(\sigma(j)=\pi(i)\). We have proved that for \(n\) sufficiently large, the largest 2-set-intersecting families in \(S_n\) are cosets of stabilizers of 2-sets.

If we stipulate that the first of the two possibilities above occurs, then we are demanding that the family be 2-intersecting; it was proved in \cite{jointpaper} that if \(n\) is sufficiently large, then the largest 2-intersecting families in \(S_n\) are cosets of stabilizers of 2 points. In both problems, the extremal families consist of all permutations acting in the same way on a small set. Indeed, in many intersection problems, the extremal families consist of all objects containing a fixed copy of the relevant structure; we say that such families are of {\em kernel-type}.

K\"orner asks what happens if instead we stipulate that the second possibility occurs. We say that a family \(\mathcal{A} \subset S_n\) is {\em reversing} if for any two permutations \(\sigma,\pi\) in \(\mathcal{A}\), there exist \(\{i,j\} \in [n]^{(2)}\) such that \(\sigma(i)=\pi(j)\) and \(\sigma(j)=\pi(i)\). K\"orner \cite{korner} asks for a determination of the maximum possible size of a reversing family in \(S_n\). As one might expect, the `reversing' condition changes the answer substantially. K\"orner makes the following.
\begin{conjecture}
\label{conj:exponential}
There exists an absolute constant \(C\) such that for any \(n \in \mathbb{N}\), a reversing family in \(S_n\) has size at most \(C^n\).
\end{conjecture}
As observed by K\"orner, we can easily produce a reversing family in \(S_n\) of size \(2^{\lfloor n/2 \rfloor}\): let \(m = \lfloor n/2 \rfloor\), and take all permutations fixing each of the 2-sets \(\{1,2\},\{3,4\},\ldots,\{2m-1,2m\}\). On the other hand, it is easy to show that a reversing family in \(S_n\) has size at most \(n! / 3^{\lfloor n/3 \rfloor}\). Indeed, let \(l = \lfloor n/3\rfloor\), and consider the following subgroup \(H\) of \(S_n\):
\[H = \langle (1,\ 2,\ 3)\rangle \times \langle (4,\ 5,\ 6)\rangle \times \ldots \times \langle(3l-2,\ 3l-1,\ 3l)\rangle .\]
Observe that \(H\) is {\em pairwise reverse-free}, meaning that for any two permutations \(\sigma,\pi\in H\), there exists no \(\{i,j\} \in [n]^{(2)}\) such that \(\sigma(i)=\pi(j)\) and \(\sigma(j)=\pi(i)\). Clearly, the same is true of any left coset of \(H\). Hence, if \(\mathcal{A} \subset S_n\) is reversing, then \(\mathcal{A}\) contains at most one permutation from each left coset of \(H\). Since \(|H| = 3^{l} = 3^{\lfloor n/3 \rfloor}\), it follows that \(|\mathcal{A}| \leq n! / 3^{\lfloor n/3 \rfloor}\), as claimed.

We write \(F(n)\) for the maximum possible size of a reversing family of permutations in \(S_n\). The best known bounds on \(F(n)\) are due to F\"uredi, Kantor, Monti and Sinaimeri \cite{furedi}:
\[(1.515\ldots-o(1))^n \leq F(n) \leq \frac{n!}{(1.686\ldots+o(1))^n}.\]

These bounds come from `product-type' constructions similar to those above. Interestingly, there is no natural candidate for an extremal family. Conjecture \ref{conj:exponential} remains open.

The above problems all concern permutations on an `unstructured' ground set. Imposing a structure on the ground set leads to other interesting extremal problems, due to K\"orner, Simonyi and Malvenuto.

We say that two permutations \(\sigma,\pi \in S_n\) {\em collide} if there exists \(i \in [n]\) such that \(|\sigma(i)-\pi(i)| = 1\), i.e. \(i\) is mapped to two consecutive numbers by \(\sigma\) and \(\pi\). We say that a family \(\mathcal{A} \subset S_n\) is {\em pairwise colliding} if any two permutations in \(\mathcal{A}\) collide. Let \(\rho(n)\) denote the maximum possible size of a pairwise colliding family of permutations in \(S_n\); K\"orner and Malvenuto \cite{malvenuto} make the following.
\begin{conjecture}
\label{conj:colliding}
For any \(n \in \mathbb{N}\), \(\rho(n) = {n \choose \lfloor n/2 \rfloor}\).
\end{conjecture}
The upper bound \(\rho(n) \leq {n \choose \lfloor n/2 \rfloor}\) can be seen as follows. If \(\sigma \in S_n\), we define the {\em parity-sequence} of \(\sigma\) to be the length-\(n\) sequence \((\epsilon_i)_{i=1}^{n}\) with
\[\epsilon_i = \left\{\begin{array}{rl} 1 & \textrm{if }\sigma(i)\ \textrm{is even;}\\
                       -1 & \textrm{if } \sigma(i)\ \textrm{is odd.}\end{array}\right.\]
There are \({n \choose \lfloor n/2 \rfloor}\) possible parity-sequences. Observe that each permutation in a pairwise colliding family must have a different parity-sequence; this yields the upper bound. Conjecture \ref{conj:colliding} has been verified for \(n \leq 7\) by K\"orner and Malvenuto \cite{malvenuto} (using ingenious combinatorial arguments), and for \(n=8,9\) by Brik \cite{brik} (using highly a non-trivial computer program), but remains open for all \(n >9\). The best known lower bound is \(\rho(n) \geq c^n\), where \(c = 1.8155...\), due to Brightwell, Cohen, Fachini, Fairthorne, K\"orner, Simonyi, and T\'oth \cite{brightwell}.

In \cite{graphdifferent}, K\"orner, Simonyi and Sinaimeri study a broad generalization of the above problem. Let \(G\) be an infinite graph with vertex-set \(\mathbb{N}\); we think of \(G\) as fixed. We say that two permutations \(\sigma,\pi \in S_n\) are {\em \(G\)-different} if there exists some \(i \in [n]\) such that \(\{\sigma(i),\pi(i)\} \in E(G)\), i.e. there is some number which \(\sigma\) and \(\pi\) map to adjacent vertices of \(G\). We say that a family of permutations \(\mathcal{A} \subset S_n\) is {\em pairwise \(G\)-different} if any two permutations \(\sigma\) and \(\pi\) in \(\mathcal{A}\) are \(G\)-different, and we write \(T(n,G)\) for the maximum possible size of a pairwise \(G\)-different family of permutations in \(S_n\).

A natural class of infinite graphs to consider is the class of distance graphs. If \(D \subset \mathbb{N}\), the {\em distance graph} \(G(D)\) on \(\mathbb{N}\) is the graph with vertex-set \(\mathbb{N}\), where we join \(i\) and \(j\) if \(|i-j| \in D\). (Note that the case \(S=\{1\}\) corresponds to colliding permutations.) In \cite{graphdifferent}, K\"orner, Simonyi and Sinaimeri obtain the value of \(T(n,G(\overline{\{q\}}))\) for all integers \(n\) and \(q\), and exhibit sets \(D\) for which \(T(n,G(D))\) has several different growth-rates. For more results and problems in this area, the reader is encouraged to consult \cite{brightwell} and \cite{graphdifferent}.

\subsection*{Erratum, added 5th December 2019}

Recently, Yuval Filmus \cite{filmus-comment} has pointed out that Theorem \ref{thm:bfp} is false for each $t \geq 2$ (there is a hole in the proof). However, it is easy to deduce Theorem \ref{thm:equality} from Theorem 3 in \cite{eff3}, which states the following.

\begin{theorem}
\label{thm:eff3}
For each \(t \in \mathbb{N}\), there exists \(C_t >0\) such that the following holds. Let $n \in \mathbb{N}$ with $n \geq t$, and let \(\mathcal{A} \subset S_n\) with \(|\mathcal{A}| = c(n-t)!\), where $c \geq 0$. Let \(f = 1_{\mathcal{A}} \colon S_n \to \{0,1\}\) denote the characteristic function of \(\mathcal{A}\), so that \(\mathbb{E}[f] = c/(n)_t\). Let $f_t$ denote the orthogonal projection of $f$ onto $V_t$. Let $0 \leq \epsilon \leq 1$. If \(\mathbb{E}[(f-f_t)^2] \leq \epsilon c/(n)_t\), then there exists $\mathcal{C} \subset S_n$ such that $\mathcal{C}$ is a union of \(\round c\) $t$-cosets of \(S_n\), and
\begin{equation}\label{eq:main-bound} |\mathcal{A} \triangle \mathcal{C}| \leq C_t(\epsilon^{1/2} + c/\sqrt{n})|\mathcal{A}|.\end{equation}
Moreover, we have $|c-\round c| \leq C_t(\epsilon^{1/2} + c/\sqrt{n})c.$
\end{theorem}
(Here, if $g:S_n \to \mathbb{R}$, $\mathbb{E}[g]$ denotes the expectation of $g$ with respect to the uniform measure on $S_n$, and if $c \geq 0$, $\round c$ denotes the closest integer to $c$, rounding up if $c+\tfrac{1}{2} \in \mathbb{N}$.)

Here is a sketch of the deduction of Theorem \ref{thm:equality} from the above theorem. Let $\mathcal{A} \subset S_n$ be a $t$-setwise-intersecting family with $|\mathcal{A}| = t!(n-t)!$. Provided $n$ is sufficiently large depending on $t$, as argued above, it follows that $1_{\mathcal{A}} \in U_t \leq V_t$, so we may apply Theorem \ref{thm:eff3} with $\epsilon=0$, concluding that there exists a family $\mathcal{C} \subset S_n$ which is a union of $t!$ $t$-cosets of $S_n$, and such that 
\begin{equation}\label{eq:symm-diff-bound} |\mathcal{A} \triangle \mathcal{C}| \leq O_t(1 /\sqrt{n})(n-t)!.\end{equation}
Without loss of generality, by translating $\mathcal{A}$ on the left and the right by fixed permutations if necessary, we may assume that $\mathcal{C}$ contains all the permutations fixing $[t]$ pointwise. It is easy to check that if $\mathcal{A}$ contains at least one permutation ($\tau$, say) that does not fix $[t]$ as a set, then there are $\Theta_t(1)((n-t)!)$ permutations that fix $[t]$ pointwise and do not agree with $\tau$ on any $t$-set, so that $|\mathcal{C} \setminus \mathcal{A}| \geq \Theta_t(1)((n-t)!)$, contradicting (\ref{eq:symm-diff-bound}) provided $n$ is sufficiently large depending on $t$. 

\subsection*{Acknowledgements}
We thank J\'anos K\"orner for helpful suggestions, and for posing the $t=2$ case of the problem considered here. We thank Ferdinand Ihringer and Karen Meagher for pointing out the second `exceptional' case of equality ($t=2$, $n=5$) in Conjecture \ref{conj:main}.

\end{document}